\documentclass[12pt,twoside]{article}
\usepackage{amsmath,amsfonts,amssymb,a4,enumitem,epsfig,array,psfrag}
\usepackage{url}
\usepackage{amsthm}
\usepackage{etoolbox} 

\usepackage{subfig}
\usepackage{chngcntr} 
\usepackage{color}
\usepackage[section]{placeins}

\newtheorem{theo}{Theorem}
\newtheorem{prop}{Proposition}[section]
\newtheorem{coro}[prop]{Corollary}
\newtheorem{lemma}[prop]{Lemma}
\newtheorem{conj}[prop]{Conjecture}

\theoremstyle{definition}

\newtheorem{example}[prop]{Example}
\AtEndEnvironment{example}{\null\hfill$\diamond$}%

\AtEndEnvironment{remark}{\null\hfill$\diamond$}%

\newcommand{\bft}{{\mathbf t}}

\newcommand{\bt}{{\mathbf t}}

\newcommand{\bT}{{\mathbf T}}
\newcommand{\bX}{{\mathbf X}}

\newcommand{\bfp}{{\mathbf p}}
\newcommand{\emptyplug}{{\bfp_\circ}}
\newcommand{\fullplug}{{\bfp_\bullet}}
\newcommand{\plug}{\operatorname{plug}}

\newcommand{\Prob}{\operatorname{Prob}}

\newcommand{\nvert}{\operatorname{vert}}

\newcommand{\interior}{\operatorname{int}}

\newcommand{\NN}{{\mathbb{N}}}
\newcommand{\ZZ}{{\mathbb{Z}}}
\newcommand{\QQ}{{\mathbb{Q}}}
\newcommand{\RR}{{\mathbb{R}}}
\newcommand{\CC}{{\mathbb{C}}}
\newcommand{\Ss}{{\mathbb{S}}}

\newcommand{\EE}{{\mathbb{E}}}

\newcommand{\cT}{{\cal T}}

\newcommand{\cF}{{\cal F}}

\newcommand{\cC}{\mathcal{C}}
\newcommand{\cD}{{\cal D}}
\newcommand{\cP}{{\cal P}}
\newcommand{\cR}{{\cal R}}
\newcommand{\cG}{{\cal G}}

\newcommand{\Hom}{\operatorname{Hom}}
\newcommand{\Tw}{\operatorname{Tw}}

\begin{document}
\title{Domino tilings of cylinders: \\
connected components under flips \\
and normal distribution of the twist}
\author{Nicolau C. Saldanha}

\maketitle

\begin{abstract}
We consider domino tilings of $3$-dimensional cubiculated regions.
A three-dimensional domino is a $2\times 1\times  1$
rectangular cuboid.
We are particularly interested in regions of the form
$\cR_N = \cD \times [0,N]$ where
$\cD \subset \RR^2$ is a fixed quadriculated disk.
In dimension $3$,
the twist associates to each tiling $\bt$ an integer $\Tw(\bt)$.
We prove that, when $N$ goes to infinity,
the twist follows a normal distribution.

A flip is a local move: two neighboring parallel dominoes
are removed and placed back in a different position.
The twist is invariant under flips.
A quadriculated disk $\cD$ is regular if,
whenever two tilings $\bt_0$ and $\bt_1$ of $\cR_N$
satisfy $\Tw(\bt_0) = \Tw(\bt_1)$,
$\bt_0$ and $\bt_1$ can be joined by a sequence of flips
provided some extra vertical space is allowed.
Many large disks are regular, including rectangles 
$\cD = [0,L] \times [0,M]$ with $LM$ even and $\min\{L,M\} \ge 3$.
For regular disks, we describe the larger connected components
under flips of the set of tilings of the region
$\cR_N = \cD \times [0,N]$.
As a corollary,
let $p_N$ be the probability that two random tilings
$\bT_0$ and $\bT_1$ of $\cD \times [0,N]$
can be joined by a sequence of flips
conditional to their twists being equal.
Then $p_N$ tends to $1$ if and only if $\cD$ is regular.


Under a suitable equivalence relation,
the set of tilings has a group structure,
the {\em domino group} $G_{\cD}$.
These results illustrate the fact that
the domino group dictates many properties of the space of tilings
of the cylinder $\cR_N = \cD \times [0,N]$, particularly for large $N$.
\end{abstract}



\section{Introduction}

A quadriculated region $\cD \subset \RR^2$
is a {\em planar quadriculated disk}
if $\cD$ is the union of finitely many closed unit squares
with vertices in $\ZZ^2$ and $\cD$ is homeomorphic to the closed unit disk.
We always assume that our quadriculated regions $\cD$ are {\em balanced}
(equal number of white and black unit squares;
squares come with a checkerboard coloring).
A quadriculated disk $\cD$ is {\em nontrivial}
if it has at least $6$ unit squares and
at least one square has at least three neighbours.
A {\em domino} is the union of two adjacent unit squares.
A {\em (domino) tiling} of $\cD$ is a set of dominoes
with disjoint interiors whose union is $\cD$;
the set of domino tilings of $\cD$ is denoted by $\cT(\cD)$.
From a graph-theoretical perpective,
we identify $\cD$ with a simple graph whose vertices are unit squares
(or, alternatively, their centers)
and whose edges are dominoes
(or, alternatively, unit segments joining the centers of two adjacent squares).
In this language, a tiling is a perfect matching.

Similarly, a cubiculated region is a set $\cR \subset \RR^3$ 
which is a union of finitely many unit cubes with vertices in $\ZZ^3$.
A  {\em (3D) domino} is the union of two adjacent unit cubes
and a {\em tiling} (of $\cR$) is a set of dominoes with disjoint interiors
whose union is $\cR$.
A {\em cylinder} is a cubiculated region of the form
$\cR_N = \cD \times [0,N]$ where $\cD$ is a quadriculated disk.
Let $\cT(\cR)$ be the set of domino tilings of $\cR$.
Again, $\cR$ can be identified with a bipartite graph,
dominoes with edges and tilings with perfect matchings.
We follow \cite{primeiroartigo,segundoartigo,regulardisk}
in drawing tilings of cubiculated regions by floors,
as in Figure \ref{fig:flip}.
Vertical dominoes (i.e., dominoes in the $z$ direction)
appear as two squares; we leave the right square unfilled.

A flip is a local move in either $\cT(\cD)$ or $\cT(\cR)$
from one tiling to another:
remove two adjacent and parallel dominoes and
place them back in a different position.
An example of a flip is shown in Figure \ref{fig:flip}.
If $\cD \subset \RR^2$ is a quadriculated disk
then $\cT(\cD)$ is connected under flips
(\cite{thurston1990}, \cite{saldanhatomei1995}).
For $\bt_0, \bt_1 \in \cT(\cR)$, we write $\bt_0 \approx \bt_1$
if $\bt_0$ and $\bt_1$ are in the same connected component under flips,
i.e., if they can be joined by a finite sequence of flips.
One of our aims is to study the connected components
of $\cT(\cR)$ under flips.

\begin{figure}[ht]
\centering
\def\svgwidth{120mm}
\begingroup%
  \makeatletter%
  \providecommand\color[2][]{%
    \errmessage{(Inkscape) Color is used for the text in Inkscape, but the package 'color.sty' is not loaded}%
    \renewcommand\color[2][]{}%
  }%
  \providecommand\transparent[1]{%
    \errmessage{(Inkscape) Transparency is used (non-zero) for the text in Inkscape, but the package 'transparent.sty' is not loaded}%
    \renewcommand\transparent[1]{}%
  }%
  \providecommand\rotatebox[2]{#2}%
  \newcommand*\fsize{\dimexpr\f@size pt\relax}%
  \newcommand*\lineheight[1]{\fontsize{\fsize}{#1\fsize}\selectfont}%
  \ifx\svgwidth\undefined%
    \setlength{\unitlength}{1276.21492748bp}%
    \ifx\svgscale\undefined%
      \relax%
    \else%
      \setlength{\unitlength}{\unitlength * \real{\svgscale}}%
    \fi%
  \else%
    \setlength{\unitlength}{\svgwidth}%
  \fi%
  \global\let\svgwidth\undefined%
  \global\let\svgscale\undefined%
  \makeatother%
  \begin{picture}(1,0.15695206)%
    \lineheight{1}%
    \setlength\tabcolsep{0pt}%
    \put(0,0){\includegraphics[width=\unitlength,page=1]{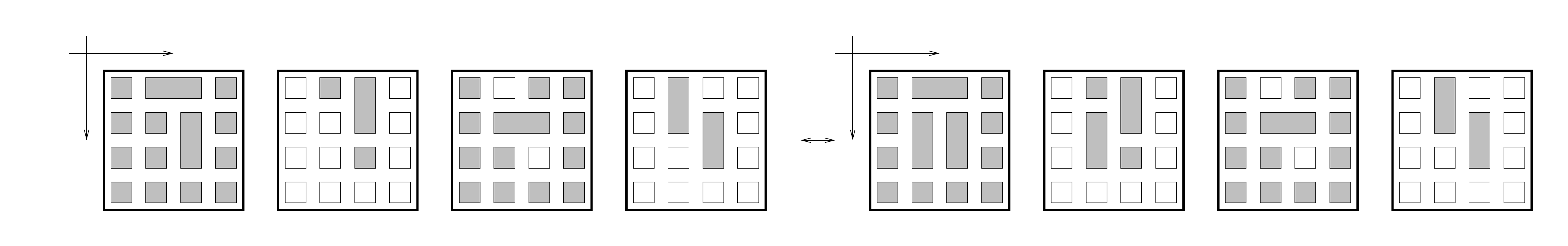}}%
    \put(0.02189859,0.06737734){\color[rgb]{0,0,0}\makebox(0,0)[lt]{\lineheight{1.25}\smash{\begin{tabular}[t]{l}$x$\end{tabular}}}}%
    \put(0.11741356,0.12957313){\color[rgb]{0,0,0}\makebox(0,0)[lt]{\lineheight{1.25}\smash{\begin{tabular}[t]{l}$y$\end{tabular}}}}%
  \end{picture}%
\endgroup%

\caption{Two tilings of the box $[0,4]\times [0,4]\times [0,4]$,
joined by a flip.}
\label{fig:flip}
\end{figure}

For a fixed balanced quadriculated disk $\cD$,
let $\cR_N$ be the cylinder $\cD \times [0,N]$.
Two tilings $\bt_0 \in \cT(\cR_{N_0})$ and $\bt_1 \in \cT(\cR_{N_1})$
can be concatenated to define a tiling
$\bt_0 \ast \bt_1 \in \cT(\cR_{N_0+N_1})$.
For $N$ even, there exists a tiling $\bt_{\nvert,N} \in \cT(\cR_N)$
with all dominoes vertical.
Assume $N_0 \equiv N_1 \pmod 2$ and $\bt_i \in \cT(\cR_{N_i})$;
we write $\bt_0 \sim \bt_1$ if and only if there exists $N \equiv N_0 \pmod 2$,
$N \ge \max\{N_0,N_1\}$ such that
we have $\bt_0 \ast \bt_{\nvert,N-N_0} \approx \bt_1 \ast \bt_{\nvert,N-N_1}$.
Given a balanced quadriculated disk $\cD$,
we define the {\em (full) domino group $G_{\cD}$}:
its elements are $\sim$-equivalence classes of tilings
and its operation is $\ast$, the concatenation.
In \cite{regulardisk,4domino} we construct a finite $2$-complex $\cC_{\cD}$
for which $G_{\cD} = \pi_1(\cC_{\cD})$ is the fundamental group,
and therefore finitely presented.
We review the construction of the domino group
in Section \ref{section:review}.

Given a tiling $\bt$ of $\cR_N$,
we define in \cite{FKMS, primeiroartigo,segundoartigo,regulardisk}
the {\em twist} $\Tw(\bt) \in \ZZ$.
A general definition is too complicated to be given here.
Example \ref{example:twist} gives a definition of the twist for cylinders.
One of the most fundamental properties of the twist
is that it is preserved by flips,
so that $\bt_0 \approx \bt_1$ implies $\Tw(\bt_0) = \Tw(\bt_1)$.
There are other local moves
(such as the {\em trit})
which change the twist in a predictable way.
We have $\Tw(\bt_{\nvert,N}) = 0$ (for all $N$) and
$\Tw(\bt_0 \ast \bt_1) = \Tw(\bt_0) + \Tw(\bt_1)$.
The {twist}
is therefore a group homomorphism $\Tw: G_{\cD} \to \ZZ$.
This homomorphism is surjective for nontrivial $\cD$.
For most of the paper we only need these properties.

A quadriculated disk $\cD$ is {\em regular} if and only if
for any tilings $\bt_0, \bt_1$ of $\cR_N$
$\Tw(\bt_0) = \Tw(\bt_1)$ implies $\bt_0 \sim \bt_1$
(but not necessarily $\bt_0 \approx \bt_1$).
It is not hard to check that $\cD$ is regular
if and only if $G_{\cD} \approx \ZZ \oplus (\ZZ/(2))$
(the first coordinate is the twist,
the second one is the parity of $N$).
A vague conjecture is that ``large'' quadriculated disks are regular.
In Theorems \ref{theo:rectangle} and \ref{theo:M}
we restate the main results of \cite{regulardisk}
with some extra information (which is also proved in that paper).
Here, $F_2$ is the free group in $2$ generators $a$ and $b$
and $F_2 \ltimes \ZZ/(2)$ is the semidirect product
defined by the involution $a \mapsto b^{-1} \mapsto a$.

\begin{theo}
\label{theo:rectangle}\cite{regulardisk}
Let $\cD = [0,L] \times [0,M]$ be a rectangle with $LM$ even.
Then $\cD$ is regular if and only if $\min\{L,M\} \ge 3$.
If $L = 2$ and $M > 2$ then there exists a surjective homomorphism
$G_{\cD} \to F_2 \ltimes \ZZ/(2)$;
in particular, $G_{\cD}$ has exponential growth.
\end{theo}

\begin{theo}
\label{theo:M}\cite{regulardisk}
Let $\cD$ be a regular quadriculated disk containing a $2\times 3$ rectangle.
Then there exists $M$ (depending on $\cD$ only)
such that for all $N \in \NN$ and for all $\bt_0, \bt_1 \in \cT(\cR_N)$
if $\Tw(\bt_0) = \Tw(\bt_1)$
then $\bt_0 \ast \bt_{\nvert,M} \approx \bt_1 \ast \bt_{\nvert,M}$.
\end{theo}

Given $\cR$, a {\em random tiling} of $\cR$ is
a uniformly distributed random variable
$\bT: \Omega \to \cT(\cR)$.
The following conjecture is formulated for boxes
but should hold in greater generality.
There is some empirical evidence towards it
(but it is hard to tell how significant it is;
see Figure \ref{fig:tw60} for a sample).

\begin{conj}
\label{conj:main}
Consider the cubiculated box
$\cR = [0,L] \times [0,M] \times [0,N]$,
$L, M, N \in \NN^\ast$, $LMN$ even.
Let $\bT, \bT_0, \bT_1$ be independent random tilings of $\cR$.
\begin{enumerate}
\item{ As $\min(L,M,N) \to \infty$,
the real random variable
\[ \frac{1}{\sqrt{LMN}} \Tw(\bT) \]
converges in distribution to a normal distribution centered at $0$.}
\item{ Two tilings with the same twist can almost always
be joined by a finite sequence of flips:
\[ \lim_{\min(L,M,N) \to \infty}
\Prob[\bT_0 \approx \bT_1 | \Tw(\bT_0) = \Tw(\bT_1)]  = 1. \]
}
\end{enumerate}
\end{conj}

In this paper we prove related results for tall cylinders.
Thus, $L$ and $M$ are kept fixed while $N$ goes to infinity.
Also, the theorem is restricted to cylinders (not just boxes).
In this introduction we state simpler results,
with an emphasis on the main examples:
regular disks and the twist.
In Section \ref{section:review},
after reviewing the contents of \cite{regulardisk},
we state more general results.
The proofs are not significantly different.
Our first main result establishes normal distribution of the twist
for tilings of nontrivial cylinders.

\begin{theo}
\label{theo:normaltwist}
Let $\cD \subset \RR^2$ be a nontrivial quadriculated disk.
Let $\bT_N$ be a random tiling of $\cR_N = \cD \times [0,N]$.
There exist constants $C_0, C_1 \in (0,+\infty)$
with the property below.

If $(t_N)$ is a sequence of integers with
$\lim_{N \to \infty} t_N/\sqrt{N} = \tau \in \RR$ then
\[ \lim_{N \to \infty} \sqrt{N}\,\Prob[\Tw(\bT_N) = t_N]
= C_0 \exp(-C_1 \tau^2). \]
\end{theo}

This roughly corresponds to a special case of
the first item of Conjecture \ref{conj:main}.
The normal distribution is illustrated in Figure \ref{fig:tw60};
there is a similar figure in \cite{regulardisk}.
Convergence in distribution follows from the statement,
but is easier and proved first in Theorem \ref{theo:normal},
stated in Section \ref{section:review} and
proved in Section \ref{section:normal}.
Theorem \ref{theo:normaltwist} is a special case of
Theorem \ref{theo:truenormal},
also to be stated in Section \ref{section:review} and proved
in Section \ref{section:truenormal}.

\begin{figure}[ht]
\begin{center}
\includegraphics[scale=0.75]{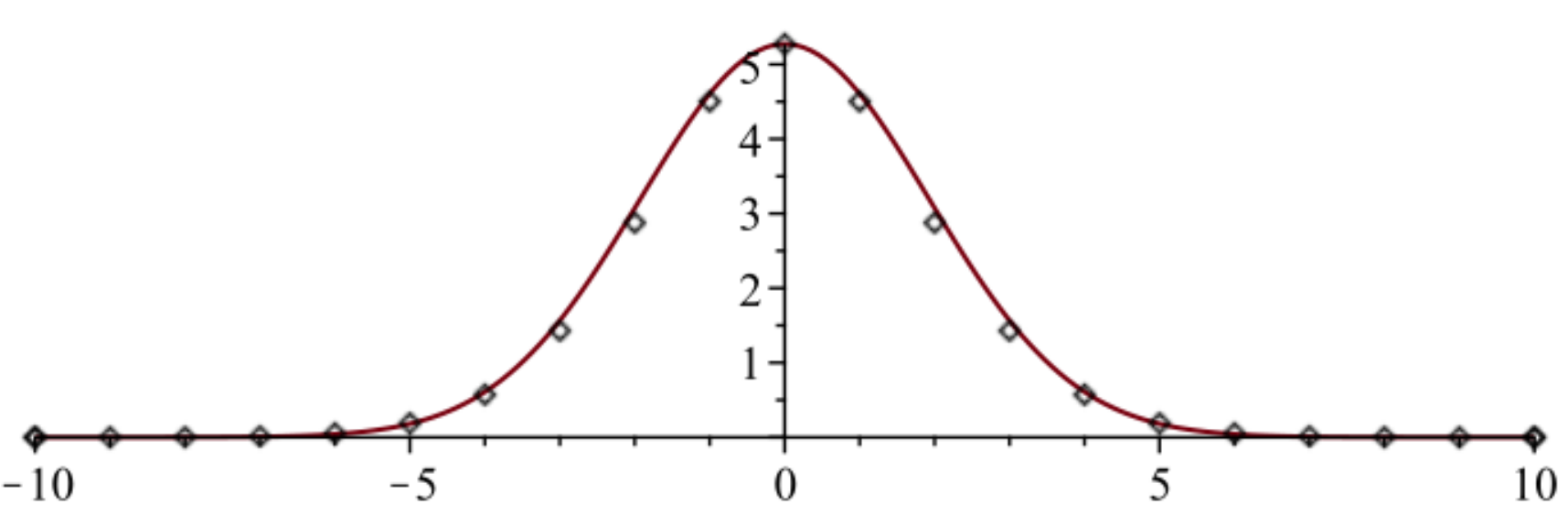}
\end{center}
\caption{The number of tilings per value of the twist
of the box $\cR_{60} = \cD \times [0,60]$, $\cD = [0,4]\times[0,4]$.
The solid curve is a true gaussian, shown for comparison.
Numbers on the vertical axis should be multiplied by $10^{156}$.}
\label{fig:tw60}
\end{figure}

The next result assumes $\cD \subset \RR^2$ to be regular.
We then give a description of $\approx$-connected components of $\cT(\cR_N)$.
We are particularly interested in the larger components.
Let $M$ be as in Theorem \ref{theo:M}.
A component $\cC \subseteq \cT(\cR_N)$ is {\em fat}
if it includes at least one tiling with at least $M$ vertical floors;
a component is {\em thin} otherwise.
Notice that vertical floors can be moved up and down via flips
(this is not hard, and follows from Lemma 5.2 in \cite{regulardisk}):
thus, in a fat component, there exists a tiling such that
the last $M$ floors are vertical.
For a component $\cC$ with $\bt \in \cC$, we write $\Tw(\cC) = \Tw(\bt)$.

\begin{theo}
\label{theo:components}
Let $\cD \subset \RR^2$ be a regular quadriculated disk.
There exists constants $c_{\cD} \in \QQ$, $b \in \RR$ and $\tilde c \in (0,1)$
such that, for all $N \in \NN$, the properties below hold.
\begin{enumerate}
\item{If $\bt$ is a tiling of $\cR_N$ then
$\Tw(\bt) \in \ZZ \cap [-c_{\cD}N, c_{\cD}N]$.}
\item{For $t \in \ZZ \cap [-c_{\cD}N+b,c_{\cD}N-b]$,
there exists exactly one fat component
$\cF_{N,t} \subseteq \cT(\cR_N)$ with $\Tw(\cF_{N,t}) = t$.}
\item{For $t \in \ZZ$, there exists at most
one fat component $\cF_{N,t}$  with $\Tw(\cF_{N,t}) = t$.}
\item{Let $\theta_N$ be the number of tilings in all thin components
of $\cT(\cR_N)$. Then, as $N$ goes to infinity,
$\theta_N = |\cT(\cR_N)| \, o(\tilde c^N)$.}
\end{enumerate}
\end{theo}

The constant $c_{\cD}$ is the same one introduced
in Section 11 of \cite{regulardisk};
see particularly Lemma 11.1
(restated below as Lemma \ref{lemma:cD}).
There are many examples of
quadriculated disks $\cD \subset \RR^2$
for which $\cT(\cR_N)$ can be shown
to have many thin connected components,
even many connected components with a single tiling.
The examples in Figure \ref{fig:single} below
and in Figure 7 in \cite{regulardisk} should make this clear.
It would be interesting to have more precise results,
perhaps including estimates for the sizes of thin connected components.

\begin{figure}[ht]
\begin{center}
\includegraphics[scale=0.27]{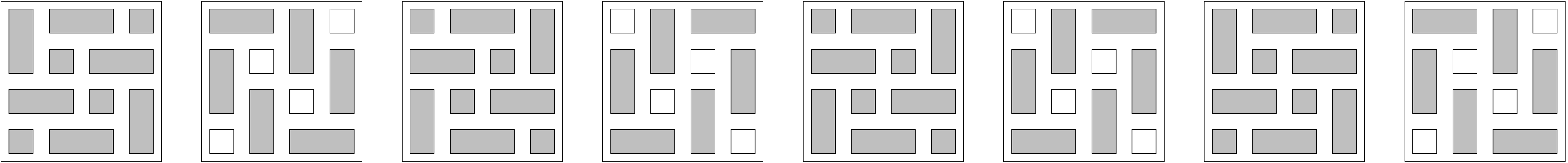}
\end{center}
\caption{A tiling of $[0,4]^2 \times [0,8]$ which admits no flips.}
\label{fig:single}
\end{figure}

Our next result has been announced (but not proved) in \cite{regulardisk}.
It follows easily from the other theorems in this paper.
Consider a nontrivial quadriculated disk $\cD \subset \RR^2$.
There is an obvious homomorphism $G_{\cD} \to \ZZ/(2)$
taking $\bt \in \cT(\cR_N)$ to $N \bmod 2$:
let $G^{+}_{\cD} < G_{\cD}$ be its kernel.
Consider the restriction $\Tw: G^{+}_{\cD} \to \ZZ$
and its kernel $\ker(\Tw) < G^{+}_{\cD}$.
We have $|\ker(\Tw)| = 1$ if and only if $\cD$ is regular.
If $\ker(\Tw)$ is infinite, $1/|\ker(\Tw)| = 0$.

\begin{theo}
\label{theo:limitprob}
Let $\cD$ be a non trivial quadriculated disk.
Let $G^{+}_{\cD}$ be the even domino group;
let $\Tw: G^{+}_{\cD} \to \ZZ$ be the twist map.
Let $\bT_0, \bT_1$ be independent random tilings
of $\cR_N = \cD \times [0,N]$;
we have
\[ \lim_{N \to \infty} \Prob[\bT_0 \approx \bT_1 | \Tw(\bT_0) = \Tw(\bT_1)]
= \frac{1}{|\ker(\Tw)|}.
\]
\end{theo}



The simplest local move after the flip is the {\em trit}.
We consider three dominoes in three different directions
whose union is a $2\times 2\times 2$ cube
minus two unit cubes in opposite corners.
Remove the three dominoes and place them back
in the only other possible way.
Two examples of trits are shown in Figure \ref{fig:trit2}.

\begin{figure}[ht]
\begin{center}
\includegraphics[scale=0.27]{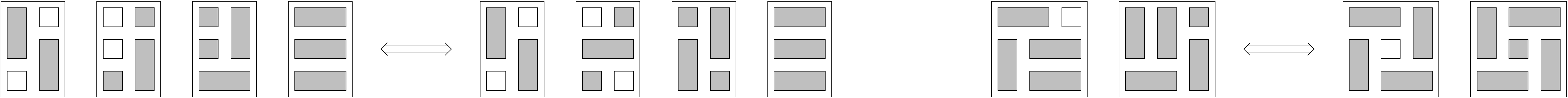}
\end{center}
\caption{Two examples of trits.
In both cases the tiling on the left has twist $0$,
the one on the right has twist $+1$.}
\label{fig:trit2}
\end{figure}

It is shown in \cite{FKMS, segundoartigo} that if two tilings
$\bt_0$ and $\bt_1$ are joined by a trit then
$\Tw(\bt_1) = \Tw(\bt_0) \pm 1$;
the sign depends on the orientation of the trit.
The following result shows that almost any two tilings 
can be connected by flips and trits.
In most cases, if $\bt_a, \bt_b \in \cT(\cR_N)$
there exists a path of flips and trits from $\bt_a$ to $\bt_b$
with precisely $|\Tw(\bt_a) - \Tw(\bt_b)|$ trits.

\begin{theo}
\label{theo:trit}
Let $\cD \subset \RR^2$ be a regular quadriculated disk.
Assume furthermore that $\cD$ contains a $2 \times 3$ rectangle.
Let $c \in \QQ$, $b \in \RR$
be as in Theorem \ref{theo:components}.
Let $\cF_{N,t}$ be the fat components under flips of $\cT(\cR_N)$.
Then there exists $\tilde b \in \RR$, $\tilde b \ge b$,
with the following property.

If $t_0 \in \ZZ \cap [-cN+\tilde b, cN-\tilde b-1]$
then there exist tilings $\bt_0 \in \cF_{N,t_0}$
and $\bt_1 \in \cF_{N,t_0+1}$ such that
$\bt_0$ and $\bt_1$ are joined by a trit.

The set $\cT(\cR_N)$ has a giant component under flips and trits
$\cG \subseteq \cT(\cR_N)$.
More precisely, there exists $\tilde c \in (0,1)$ such that
$|\cT(\cR_N) \smallsetminus \cG| = |\cT(\cR_N)|\,o(\tilde c^N)$.
\end{theo}

In Section \ref{section:review} we review some more results,
particularly from \cite{regulardisk},
and state the other main results of this paper,
Theorems \ref{theo:cofinite}, \ref{theo:normal} and \ref{theo:truenormal}.
These results are more general and imply Theorem \ref{theo:normaltwist}.
Their statements make essential use of the domino group $G_{\cD}$.
Section \ref{section:randomtilings} contains
the first results concerning random tilings.
This is closely related to studying random paths
in a finite graph and uses the Perron-Frobenius Theorem.
In Section \ref{section:cocycle} we use the language of homology
to obtain algebraic expressions for, say,
the number of tilings $\bt$ of $\cR_N$ with $\Tw(\bt) = t$
(in terms of $N$ and $t$).
Sections \ref{section:normal} and \ref{section:truenormal}
are the most technical and contain the proofs of
Theorems \ref{theo:normal} and \ref{theo:truenormal}.
Section \ref{section:components} contains
the proofs of Theorems \ref{theo:components},
\ref{theo:limitprob} and \ref{theo:trit}:
by this point, the proofs are short and simple.
Finally, Section \ref{section:final}
lists a few remarks and related open questions.

\smallskip

The author thanks Juliana Freire, Simon Griffiths, 
Caroline Klivans, Pedro Milet and Breno Pereira
for helpful conversations, comments and suggestions.
He would also like to thank the referees for
carefully reading the manuscript and offering several helpful comments.
The author is also thankful for the generous support of
CNPq, CAPES and FAPERJ (Brazil).




\section{Review and results}
\label{section:review}

In this section we review notation and results,
particularly from \cite{4domino, regulardisk}.
We also state the other main results of this paper,
which are harder to state than
Theorems \ref{theo:normaltwist} and \ref{theo:components}.
Our focus in the introduction, as in \cite{regulardisk},
was in regions of dimension $3$.
As in \cite{4domino},
the general construction below is also valid
for regions of dimension $n \ge 4$.
We then assume $\cD \subset \RR^{n-1}$ to be
a balanced, contractible, bounded cubiculated region
and write $\cR_N = \cD \times [0,N]$.
Remember however that for tilings $\bt$ of $\cR \subset \RR^n$, $n \ge 4$,
the twist $\Tw(\bt)$ is an element of $\ZZ/(2)$ (not of $\ZZ$).

A {\em plug} is a subset $p \subseteq \cD$
which is the union of an equal number of black and white unit squares
(or cubes);
a plug is sometimes confused with a set of unit squares (or cubes).
If $\cD$ has $2k = |\cD|$ unit squares,
the set $\cP$ of all plugs has $\binom{2k}{k}$ elements.
The empty set is the empty plug $\emptyplug \in \cP$
and $\cD$ is the full plug $\fullplug \in \cP$.
We say two plugs $p_0$ and $p_1$ are {\em disjoint}
if their interiors are disjoint.
Given two disjoint plugs $p_0$ and $p_1$,
$\cD_{p_0,p_1} \subseteq \cD$ is the union of all unit squares
contained in $\cD$ with interior disjoint from both $p_0$ and $p_1$.
Thus, $\cD_{p_0,p_1}$ is a balanced quadriculated region,
possibly disconnected and not necessarily tileable.
A {\em (full) floor} is a triple $f = (p_0,f^\ast,p_1)$
where $p_0$ and $p_1$ are disjoint plugs and
$f^\ast$ is a tiling of $\cD_{p_0,p_1}$.
A floor is vertical if $f^\ast = \emptyset$ (see Figure \ref{fig:vert});
in other words, a floor is vertical
if all dominoes intersecting it are vertical.
Informally, a floor is what we draw for each floor
of a tiling of $\cR_N = \cD \times [0,N]$,
as in Figure \ref{fig:flip}.
A {\em cork} is a region of the form
\[ \cR_{0,N;p_0,p_N} = \cR_N \smallsetminus
((p_0 \times [0,1]) \cup (p_N \times [N-1,N])); \]
a tiling $\bt$ of $\cR_{0,N;p_0,p_N}$ is identified with a sequence of
plugs and floors:
$\bt = (p_0,f_1,p_1,\ldots,p_{N-1},f_N,p_N)$ where,
for each $k$, we have $f_k = (p_{k-1},f^\ast_k,p_k)$,
$p_{k-1}$ and $p_k$ are disjoint 
and $f^\ast_k \in \cT(\cD_{p_{k-1},p_k})$.

\begin{figure}[ht]
\begin{center}
\includegraphics[scale=0.275]{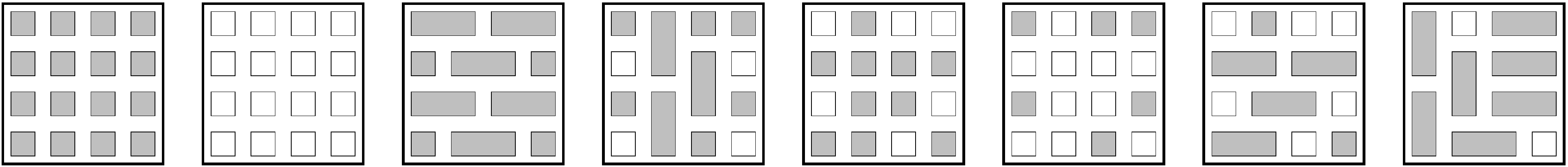}
\end{center}
\caption{A tiling $\bt$ of $\cR_8 = \cD \times [0,8]$ where $\cD = [0,4]^2$.
Floors numbers $1$, $2$, $5$ and $6$ are vertical so that $\nvert(\bt) = 4$.}
\label{fig:vert}
\end{figure}

For a tiling $\bt$ of $\cR_N$,
let $\nvert(\bt)$ be the number of vertical floors of $\bt$.
The following result is a corollary of Theorem \ref{theo:M}
and is essentially the same as Corollary 12.1 in \cite{regulardisk}.
A similar result holds for $\cD \subset \RR^{n-1}$, $n \ge 4$.

\begin{prop}
\label{prop:giant}
Let $\cD \subset \RR^2$ be a regular disk.
Let $M$ be as in Theorem \ref{theo:M}.
Let $\bt_0, \bt_1$ be tilings of $\cR_N$.
If $\Tw(\bt_0) = \Tw(\bt_1)$, $\nvert(\bt_0) \ge M$ and $\nvert(\bt_1) \ge M$
then $\bt_0 \approx \bt_1$.
\end{prop}

The set of vertices of the complex $\cC_{\cD}$ is $\cP$.
Each floor $f = (p_0,f^\ast,p_1)$ is
an undirected edge joining $p_0$ and $p_1$.
Notice that tilings of $\cD$ are loops from $\emptyplug$ to itself.
A tiling $\bt \in \cT(\cR_{0,N;p_0,p_N})$ is identified with
a path of length $N$ in $\cC_{\cD}$ from $p_0$ to $p_N$.
Let $A \in \ZZ^{\cP\times\cP}$ be the adjacency matrix of $\cC_{\cD}$:
$A_{p_0,p_1} = |\cT(\cD_{p_0,p_1})|$ if $p_0$ and $p_1$ are disjoint,
$A_{p_0,p_1} = 0$ otherwise.
We then have $|\cT(\cR_{0,N;p_0,p_N})| = (A^N)_{p_0,p_N}$.
For sufficiently large $N$,
all entries of $A^N$ are strictly positive.

Flips define the $2$-cells of $\cC_{\cD}$:
horizontal flips are bigons and vertical flips are squares.
We also need one $2$-cell for each self-loop.
If $\bt_0 \in \cT(\cR_{0,N_0;p_0,p_1})$
and $\bt_1 \in \cT(\cR_{0,N_1;p_0,p_1})$ 
are interpreted as paths in $\cC_{\cD}$,
we have $\bt_0 \sim \bt_1$ if and only if
the two paths are homotopic with fixed endpoints.
The domino group $G_{\cD} = \pi_1(\cC_{\cD};\emptyplug)$
is the set of tilings of $\cR_N$ (for all $N$)
modulo the equivalence relation $\sim$.
Vertical tilings represent the identity element in $G_{\cD}$.
The group operation is $\ast$, the concatenation of tilings (or paths).
The structure of $G_{\cD}$ is very informative.

We are ready to state another of our main results.

\begin{theo}
\label{theo:cofinite}
Let $K < G_{\cD}$ be a subgroup of finite index $n$ with
$K \not\subseteq G^{+}_{\cD}$;
let $\psi: G_{\cD} \to G_{\cD}/K$ be the quotient map.
For each $N$, consider the random variable $\bX_N = \psi(\bT)$
with values in the finite set $G_{\cD}/K$;
here $\bT$ is a random tiling of $\cR_N$.
The random variables $\bX_N$ converge in distribution
to the uniform distribution in the set $G_{\cD}/K$.
\end{theo}

The proof of Theorem \ref{theo:cofinite} is given
in Section \ref{section:randomtilings}.

For a sequence of random variables such as $(\bX_N)$,
assuming values in a fixed finite set,
the notion of convergence is easy.
Indeed, we are merely stating that, for any $x \in G_{\cD}/K$,
we have
\[ \lim_{N \to \infty} \Prob[\bX_N = x] = \frac{1}{n}. \]
We shall soon discuss the more complicated case
of random variables assuming values in infinite sets.

Notice that we do not assume that $K$ is a normal subgroup;
the quotient space $G_{\cD}/K$ in the statement above
can be taken to be either
the left coset space $\{ gK; g \in G_{\cD} \}$ or
the right coset space $\{ Kg; g \in G_{\cD} \}$
(both have $n$ elements).
The following corollary illustrates some uses of Theorem \ref{theo:cofinite}.

\begin{coro}
\label{coro:cofinite}
Consider a nontrivial quadriculated disk $\cD \subset \RR^2$
and a random tiling $\bT_N$ of $\cR_N$.
Consider a fixed positive integer $n \in \NN^\ast$.
Then, for any $a \in \ZZ$,
\begin{equation*}
\lim_{N \to \infty} \Prob[\;\Tw(\bT_N) \equiv a \pmod n\;] = \frac{1}{n},
\qquad
\lim_{N \to \infty} \Prob[\Tw(\bT_N) = a] = 0.
\end{equation*}
\end{coro}


\begin{proof}
For $n \in \NN^\ast$, let $\psi: G_{\cD} \to \ZZ/(n)$
be defined by $\psi(\bt) = \Tw(\bt) \bmod n$.  Take $K = \ker(\psi)$.
The first limit in the display
follows directly from Theorem \ref{theo:cofinite}.
For the second limit, apply the first one with $n = 2^k$
and let $k$ go to infinity.
\end{proof}


For the next theorems, random variables assume real values.
The first result uses the concept of convergence in distribution;
we shall further comment on this concept both in and after the statement.

\begin{theo}
\label{theo:normal}
Let $\cD$ be a nontrivial balanced quadriculated disk.
Let $\bT_N$ be a random tiling of $\cR_N$.
Let $G_{\cD}$ be the domino group.
Let $\psi: G_{\cD} \to \ZZ$ be a surjective homomorphism.
As $N \to \infty$, the real random variable
\[ \frac{1}{\sqrt{N}} \; \psi(\bT_N) \]
converges in distribution to a normal distribution centered at $0$.
In other words,
there exist $C_0, C_1 \in (0,+\infty)$ with the property below.
If $(t_N^{-})$ and $(t_N^{+})$ are sequences of integers with
$\lim_{N \to \infty} t_N^{-}/\sqrt{N} = \tau^{-}$,
$\lim_{N \to \infty} t_N^{+}/\sqrt{N} = \tau^{+}$
with $\tau^{-} < \tau^{+}$ 
then 
\[ \lim_{N \to \infty}
\Prob[t_N^{-} < \psi(\bT_N) < t_N^{+}]
= \int_{\tau^{-}}^{\tau^{+}} C_0 \exp(-C_1 \tau^2) d\tau. \]
\end{theo}

This result can be considered a variant of the Central Limit Theorem.
We prove it in Section \ref{section:normal}
using {\em characteristic functions}:
if $\bX$ is a real valued random variable,
set $\varphi_{\bX}(t) = \EE(\exp(it\bX))$.

Our final result is related but significantly different.

\begin{theo}
\label{theo:truenormal}
Let $\cD$, $\bT_N$, $G_{\cD}$, $\psi: G_{\cD} \to \ZZ$
and $C_0, C_1 \in (0,+\infty)$ be as in Theorem \ref{theo:normal}.
If $(t_N)$ is a sequence of integers with
$\lim_{N \to \infty} t_N/\sqrt{N} = \tau \in \RR$
then 
\[ \lim_{N \to \infty} \sqrt{N} \Prob[\psi(\bT_N) = t_N] =
C_0 \exp(-C_1 \tau^2). \]
\end{theo}

Figure \ref{fig:tw60} may count as experimental evidence
in favor of Theorems \ref{theo:normal} and \ref{theo:truenormal}.
The proof of Theorem \ref{theo:truenormal} is given
in Section \ref{section:truenormal};
it uses Theorem \ref{theo:normal}.
It might at first seem that Theorem \ref{theo:truenormal} would follow
rather directly from Theorem \ref{theo:normal},
but that does not seem to be correct.
The probability
in the statement of Theorem \ref{theo:normal}
can be written as a summation
containing roughly $\sqrt{N}$ terms of the form $\Prob[\psi(\bT_N) = t]$,
$t_N^{-} < t < t_N^{+}$.
But it could happen, at least in principle,
that some of the terms are much larger or much smaller
than their neighbors.

An important example of surjective homomorphism is
$\psi = \Tw: \cC_{\cD} \to \ZZ$ provided $\cD$ is not trivial:
Theorem \ref{theo:normaltwist}
is a corollary of Theorem \ref{theo:truenormal}, as claimed.
Theorem \ref{theo:truenormal} also implies a stronger version
of the second claim in Corollary \ref{coro:cofinite}.

Theorems \ref{theo:cofinite}, \ref{theo:normal} and \ref{theo:truenormal}
depend upon the structure of the domino group $G_{\cD}$:
we need either a subgroup or a homomorphism with certain properties.
In this, they are more general than the results in the introduction,
which are stated for regular disks only.
There are applications of the stronger results
to cases where $\cD$ is not regular,
and we saw a sample in Corollary \ref{coro:cofinite}.
Theorem \ref{theo:components}, however,
assumes full knowledge of the domino group
and is therefore not clear what could be said in greater generality.

\section{Random tilings and proof of Theorem \ref{theo:cofinite}}
\label{section:randomtilings}

It follows from
the identification between tilings of corks and
paths in the complex $\cC_{\cD}$,
discussed in the previous section,
that any question concerning random tilings
admits a translation as a question concerning random paths in $\cC_{\cD}$.
We try to give elementary proofs.

\begin{lemma}
\label{lemma:pf}
Given a quadriculated region $\cD \subset \RR^2$ 
there exist $\lambda_1 > 0$, $c \in (0,1)$ and
a unit vector $v_1 \in \RR^\cP$
with positive coordinates such that (when $N \to \infty$)
\[ |\cT(\cR_{0,N;p,\tilde p})| =
(v_1)_{p} (v_1)_{\tilde p}\;\lambda_1^N(1+o(c^N)). \]
Furthermore,
for all $\epsilon > 0$ there exists $N_\epsilon \in \NN$
such that if $p_0, p_N \in \cP$,
$j > N_\epsilon$, $N > j + N_\epsilon$ and
$\bT$ is a random tiling of $\cR_{0,N;p_0,p_N}$
then 
\[ (v_1)_{p_j}^2 - \epsilon < \Prob[ \plug_j(\bT) = p_j ]
< (v_1)_{p_j}^2 + \epsilon. \]
\end{lemma}

\begin{proof}
We can apply the Perron-Frobenius Theorem
to the irreducible matrix $A$.
Let $\lambda_1 > 0$ be the eigenvalue of largest absolute value,
with associated unit eigenvector $v_1$.
We know that $\lambda_1$ is simple,
that the coordinates of $v_1$ can be taken to be positive
and that there exists $c \in (0,1)$ such that
$|\lambda_j| < c\lambda_1$ for any other eigenvalue $\lambda_j$.
Set $\Pi = v_1 v_1^{\ast}$, the orthogonal projection
onto the line spanned by $v_1$.
Write $A = \lambda_1 (\Pi + C)$
where $\Pi C = C\Pi = 0$ and $C$ is a strong contraction:
$|Cv| < c|v|$ for all $v \in \RR^{\cP} \smallsetminus \{0\}$.
We thus have
\[ \frac{ A^N }{\lambda_1^N} = \Pi + C^N = \Pi + o(c^N),
\quad
\frac{|\cT(\cR_{0,N;p,\tilde p})|}{\lambda_1^N} =
\frac{(A^N)_{p,\tilde p}}{\lambda_1^N} =
(v_1)_{p} (v_1)_{\tilde p} + o(c^N).
 \]
The first claim is thus proved.

For the second claim, we have
\[ \Prob[ \plug_j(\bT) = p_j ] =
\frac{|\cT(\cR_{0,j;p_0,p_j})| |\cT(\cR_{0,N-j;p_j,p_N})|}%
{|\cT(\cR_{0,N;p_0,p_N})|} 
= \frac{(A^j)_{p_0,p_j}(A^{(N-j)})_{p_j,p_N} }%
{ (A^N)_{p_0,p_N}}. \]
From the first claim, if $j$ and $N-j$ are large we have
\[  \Prob[ \plug_j(\bT) = p_j ] \approx
\frac{(v_1)_{p_0} (v_1)_{p_j} \lambda_1^j \cdot
(v_1)_{p_j} (v_1)_{p_N} \lambda_1^{(N-j)}}%
{(v_1)_{p_0} (v_1)_{p_N} \lambda_1^N} 
= (v_1)_{p_j}^2, \]
as desired.
\end{proof}




We are almost ready to present the proof of Theorem \ref{theo:cofinite}.
We present the proof for the right coset space $\{ Kg; g \in G_{\cD} \}$;
the other case requires only minor adjustments.

Recall that for well behaved path-connected spaces with base point $(Z,z_0)$
there exists a natural correspondence between
subgroups $K < \pi_1(Z;z_0)$ and connected covering spaces $\Pi: Z^K \to Z$.
Given $K$, the covering is characterized by the fact that
the image of the induced map
$\pi_1(\Pi): \pi_1(Z^K) \to \pi_1(Z)$ equals $K$.
Also, the degree of the covering equals the index of $K$ in $\pi_1(Z)$.

We review the crucial point of the construction:
consider two paths $\gamma_0, \gamma_1: [0,1] \to Z$
with $\gamma_0(0) = \gamma_1(0) = z_0$
and $\gamma_0(1) = \gamma_1(1)$.
For any covering space $\tilde Z$,
the paths can be uniquely lifted to
$\tilde\gamma_0, \tilde\gamma_1: [0,1] \to \tilde Z$
with $\tilde\gamma_0(0) = \tilde\gamma_1(0) = \tilde z_0$,
the base point of $\tilde Z$.
The covering $Z^K$ is characterized by the fact that 
$\tilde\gamma_0(1) = \tilde\gamma_1(1)$ if and only if
$[\gamma_0 \ast \gamma_1^{-1}] \in K$
(where $\ast$ stands for concatenation).
There is a natural correspondence between right cosets $Kg$,
$g \in \pi_1(Z;z_0)$,
and preimages under $\Pi$ of the base point $z_0$.

\begin{proof}[Proof of Theorem \ref{theo:cofinite}]
Apply this construction in our case to define
a covering space $\Pi: \cC_{\cD}^K \to \cC_{\cD}$ of degree $n$.
By the usual construction, $\cC_{\cD}^K$ is a finite $2$-complex.
Let $\cP^K$ be the set of vertices of $\cC_{\cD}^K$:
we call these {\em lifted plugs}.
The base point $\bfp_{\circ}^{K} \in \cP^K$ of $\cC_{\cD}^K$
is a fixed preimage under $\Pi$ of
the empty plug $\emptyplug \in \cP$.

Tilings of $\cR_N$ can be lifted to paths of length $N$ in $\cC_{\cD}^K$,
starting at the base point $\bfp_{\circ}^{K} \in \cP^K$
and ending at any preimage of $\emptyplug$.
There exists a natural identification
$\Pi^{-1}[\{\emptyplug\}] \approx G_{\cD}/K$:
for a tiling $\bt \in \cT(\cR_N)$,
$\psi(\bt) \in G_{\cD}/K$ is the final point of the corresponding
lifted path  in $\cC_{\cD}^K$.

The construction in the proof of Lemma \ref{lemma:pf}
(using the Perron-Frobenius theorem)
can be applied to paths in the complex $\cC_{\cD}^K$.
Moreover, the eigenvalue of largest absolute value is the same number
$\lambda_1 > 0$ (for both $\cC_{\cD}$ and $\cC^K_{\cD}$)
and the corresponding eigenvector $v_1^K$ is the lift of $v_1$:
for $p \in \cP^K$ we have $(v_1^K)_{p} = (v_1)_{\Pi(p)}/\sqrt{n}$.

From Lemma \ref{lemma:pf},
if $p_0, p_1 \in \Pi^{-1}[\{\emptyplug\}] \approx G_{\cD}/K$
we have
\[ |\{\bt \in \cT(\cR_{0,N;p,\tilde p}), \psi(\bt) = p_i\}| =
\frac{1}{n}\;(v_1)_{\emptyplug}\;\lambda_1^N(1+o(c^N)) \]
for both $i = 0$ and $i = 1$.
Thus
\[ \Prob[\psi(\bT) = p_0] = \frac{1}{n}\;(1+o(c^N)) \]
for all $p_0 \in G_{\cD}/K$,
completing the proof.
\end{proof}

The reader will recall from \cite{regulardisk}
that vertical floors are very useful in our constructions.
Let $\nvert(\bt)$ be the number of vertical floors of $\bt$.
The following result shows that for almost all tilings
vertical floors are relatively abundant.

\begin{lemma}
\label{lemma:manyverticalfloors}
Consider $\cD$ fixed; let $\bT$ be a random tiling of $\cR_N$.
There exist positive constants $C, c \in (0,1)$ such that
\[ \Prob[\nvert(\bT) < CN] = o(c^N). \]
\end{lemma}

In the proof below the reader should keep in mind
that we are not trying to give good estimates of the constants
$C$ and $c$.

\begin{proof}
Let $p_\bullet = \cD \in \cP$ be the plug with all squares marked.
Let $v_1 \in \RR^{\cP}$ be the unit vector with positive coordinates
introduced in Lemma \ref{lemma:pf}.
Consider $\epsilon = \frac12 \min_{p \in \cP} (v_1)_p^2 > 0$.
Again from Lemma \ref{lemma:pf}, 
let $N_\epsilon$ be such that
for all $j > N_\epsilon$, $\tilde N > j+N_\epsilon$
and for all $p_0, p_{\tilde N} \in \cP$ and for $\tilde\bT$
a random tiling of $\cT_{\tilde N;p_0, p_{\tilde N}}$ we have
$\Prob[p_j(\tilde\bT) = p_\bullet] > \epsilon$.
Take $C = \epsilon/(2N_\epsilon)$.

We imagine that $\bT$ is created floor by floor
(with the correct conditional probabilities).
Let $N_j = jN_\epsilon$.
Assuming $\bT$ constructed up to floor $N_{j-1}$
we have $\Prob[\plug_{N_j}(\bT)=p_\bullet|\cdots] > \epsilon$
(where the dots stand for the description of
the already constructed part of $\bT$).
We thus have at worst $\lfloor N/N_\epsilon \rfloor$ floors,
each with conditional probability at least $\epsilon$ of being vertical.
The claim now follows.
\end{proof}


\section{Cocycles}
\label{section:cocycle}

Let $\cD$ be a nontrivial quadriculated disk.
Let $\psi: G_{\cD} \to \ZZ$ be a surjective homomorphism
so that $\psi \in \Hom(\pi_1(\cC_{\cD});\ZZ)$.
It is a well known fact that there exists a natural isomorphism
between $\Hom(\pi_1(\cC_{\cD});\ZZ)$ and the cohomology space $H^1(\cC_{\cD})$
so that we may interpret $\psi$ as an element of $H^1(\cC_{\cD})$.
We review this construction, which will be important for us in any case.

Let $K < G_{\cD}$ be the kernel of $\psi$.
As reviewed in the previous section,
let $\Pi: \cC_{\cD}^K \to \cC_{\cD}$ be the covering space
corresponding to $K$.
The space $\cC_{\cD}^K$ is an infinite (but locally finite) $2$-complex;
let $\cP^K$ be its set of vertices
(plugs with an extra discrete information).
The map $\psi$ induces a natural correspondence between $\ZZ$
and the set of preimages $(\bfp_{\circ}^{(k)})_{k \in \ZZ}$ of 
the base point $\emptyplug \in \cC_{\cD}$.
Indeed, a tiling $\bt \in \cT(\cR_N)$ defines a closed path in $\cC_{\cD}$
which lifts to a path in $\cC_{\cD}^K$
from its base point $\bfp_{\circ}^{(0)} = \bfp_{\circ}^{K}$
to some $p \in \cP^K$ with $\Pi(p) = \emptyplug$:
if $\psi(\bt) = k$, set $p = \bfp_{\circ}^{(k)}$.
Notice that this definition is consistent:
$\psi(\bt_0) = \psi(\bt_1)$ if and only if $\psi(\bt_0 \ast \bt_1^{-1}) = 0$
if and only if $[\bt_0 \ast \bt_1^{-1}] \in K$
if and only if the endpoints of
the lifted paths to $\cC_{\cD}^K$ coincide.

More generally, there is an action of $\ZZ$ on $\cC_{\cD}^K$
generated by the deck transformation $\sigma: \cC_{\cD}^K \to \cC_{\cD}^K$
satisfying $\sigma(\bfp_{\circ}^{(0)}) = \bfp_{\circ}^{(1)}$.
Recall that a deck transformation of the covering map 
$\Pi: \cC_{\cD}^K \to \cC_{\cD}$ 
is a map $\Psi: \cC_{\cD}^K \to \cC_{\cD}^K$ satisfying
$\Pi \circ \Psi = \Pi$.
By construction, we have
$\sigma(\bfp_{\circ}^{(k)}) = \bfp_{\circ}^{(k+1)}$ (for all $k \in \ZZ$).
Let $\tilde f_0 = (p_{00}, f_0^{\ast}, p_{01})$ and
$\tilde f_1 = (p_{10}, f_1^{\ast}, p_{11})$
be oriented edges of $\cC_{\cD}^K$
so that $p_{ij} \in \cP^K$.
If $\Pi(\tilde f_0) = \Pi(\tilde f_1)$ then
there exists $k \in \ZZ$ such that
$\tilde f_1 = \sigma^k(\tilde f_0)$
(so that, in particular, $p_{1j} = \sigma^k(p_{0j})$).

We construct a function $\xi: \cP^K \to \RR$ such that,
for every $p \in \cP^K$ and every $k \in \ZZ$ we have 
$\xi(\sigma^k(p)) - \xi(p) = k$.
Indeed,
for each $p \in \cP$ choose a $\Pi$-preimage $\tilde p \in \cP^K$
and arbitrarily define $\xi(\tilde p) \in \RR$.
Next, for all $k \in \ZZ$ define
$\xi(\sigma^{k}(\tilde p)) = \xi(\tilde p) + k$,
completing the construction of $\xi$.
Notice that if $\tilde f_0 = (p_{00}, f_0^{\ast}, p_{01})$ and
$\tilde f_1 = (p_{10}, f_1^{\ast}, p_{11})$ are edges of $\cC_{\cD}^K$
and satisfy
$\Pi(\tilde f_0) = \Pi(\tilde f_1)$
then $\xi(p_{01}) - \xi(p_{00}) = \xi(p_{11}) - \xi(p_{10})$.
For any oriented edge $f$ of $\cC_{\cD}$, define
$d\xi(f) = \xi(p_1) - \xi(p_0)$.
The function $d\xi$ is well defined.
In the language of cohomology,
$d\xi \in Z^1(\cC_{\cD}) \subset C^1(\cC_{\cD})$
is a cocycle and
defines the same element of the cohomology space $H^1(\cC_{\cD})$
as the original $\psi$ under the natural isomorphism mentioned above.
Also, if $\xi_0$ and $\xi_1$ are different functions
obtained by the construction above
then $\xi_1 - \xi_0$ is a well defined function from $\cP$ to $\RR$
so that $d\xi_1 - d\xi_0 = d(\xi_1-\xi_0) \in B^1(\cC_{\cD})$,
consistently with the fact that they define the same element
of the cohomology $H^1(\cC_{\cD})$.
The function $\xi$ can be constructed so as to assume values
in $\frac1m \ZZ$ (for some positive integer $m$) or even in $\ZZ$.

\begin{example}
\label{example:twist}
For $\psi = \Tw$ we can take
$\psi_{\xi}((p,f,\tilde p)) = \tau^u(f;p,\tilde p) \in \frac14 \ZZ$
as defined in \cite{regulardisk}, Section 6, just above Lemma 6.1.
Here the vector $u \in \{\pm e_1, \pm e_2\} \subset \RR^3$
is arbitrary but fixed;
$u = e_1$ and $u = e_2$ correspond to different functions $\xi$.

Alternatively, we know from Lemma 3.1 
in \cite{regulardisk} that for every plug $p$
there exists a tiling $\bt_p \in \cT(\cR_{0,2|\cD|;\emptyplug,p})$:
fix one such tiling for every plug $p$
(recall that $|\cD| \in 2\NN$ is the number of squares of $\cD$).
Given a floor $f = (p_0,f^\ast,p_1)$,
let $\bt_f \in \cT(\cR_{4|\cD|+1})$ be obtained
by concatenating $\bt_{p_0}$, $f$ and $\bt_{p_1}^{-1}$.
We can then take
$\psi_{\xi}(f) = \Tw(\bt_f) \in \ZZ$,
thus defining a cocycle $\psi_{\xi} \in Z^1(\cC_{\cD})$
and the function $\xi$.

For either construction, if $\bt \in \cT(\cR_N)$,
$\bt = (p_0,f_1,p_1,\ldots,f_N,p_N)$
(with $p_0 = p_N = \emptyplug$)
then 
\begin{equation}
\label{equation:twistcocycle}
\Tw(\bt) = \sum_{0<k\le N} \psi_{\xi}(f_k). 
\end{equation}
This is consistent with the formulas for the twist seen,
for instance, in \cite{regulardisk}.
\end{example}

Assume below for simplicity that $\xi$ and $d\xi$
assume values in $\frac1m \ZZ$.
Let $\tilde q$ be a formal variable; let $q = \tilde q^m$.
Let $\ZZ[\tilde q, \tilde q^{-1}]$
be the ring of Laurent polynomials in the variable $\tilde q$.
We define the {\em $\psi$-adjacency matrix} or
{\em $\xi$-adjacency matrix}
$\alpha \in (\ZZ[\tilde q, \tilde q^{-1}])^{\cP\times\cP}$.
For disjoint $p, \tilde p \in \cP$, let
\[ \alpha_{(p,\tilde p)} =
\sum_{f \in \cD_{p,\tilde p}} q^{d\xi((p,f,\tilde p))} =
\sum_{f \in \cD_{p,\tilde p}} \tilde q^{(m\,d\xi((p,f,\tilde p)))}
\in \ZZ[\tilde q, \tilde q^{-1}]; \]
if $p$ and $\tilde p$ are not disjoint,
define $\alpha_{(p,\tilde p)} = 0$.
Alternatively, we may regard $\alpha$ as a function
$\alpha: \CC \smallsetminus \{0\} \to \CC^{\cP\times\cP}$
where $\alpha(z)$ is obtained from $\alpha$
by substituting $z$ for $\tilde q$;
notice that $\alpha(1) = A \in \ZZ^{\cP\times\cP}$
is the adjacency matrix of $\cC_{\cD}$,
as defined in Section \ref{section:review}.
Also,
\begin{equation}
\label{equation:P}
(\alpha^N)_{(\emptyplug,\emptyplug)} = P_N(q) =
\sum_{\bft \in \cT(\cR_N)} q^{\psi(\bft)} \in
\ZZ[q,q^{-1}] \subset \ZZ[\tilde q, \tilde q^{-1}]. 
\end{equation}
The Laurent polynomial $P_N$ will be used again.
An element $P \in \ZZ[\tilde q, \tilde q^{-1}]$, $P \ne 0$,
is a {\em monomial} if it has precisely one nonzero coefficient.

\begin{lemma}
\label{lemma:supercork}
Let $\cD$ be a nontrivial quadriculated disk.
Let $A \in \ZZ^{\cP\times\cP}$ be the adjacency matrix of $\cC_{\cD}$.
For a surjective homomorphism $\psi: G_{\cD} \to \ZZ$,
let $\alpha \in (\ZZ[\tilde q,\tilde q^{-1}])^{\cP\times\cP}$
be the $\psi$-adjacency matrix constructed above.

If $z \in \Ss^1 \subset \CC$, $N \in \NN$ and $p, \tilde p \in \cP$
then
$|((\alpha(z))^N)_{p,\tilde p}| \le |(A^N)_{p,\tilde p}|$.
Furthermore, there exists $N_0$
(depending on $\cD$ and $\psi$ only)
such that if $N \ge N_0$ then
$(\alpha^N)_{p,\tilde p} \in \ZZ[\tilde q, \tilde q^{-1}]$
is neither zero nor a monomial;
moreover, there exists $k \in \ZZ$ such that the coefficients
of $\tilde q^k$ and $\tilde q^{(k+m)}$ are both positive.
In this case, if $z = \exp(\frac{it}{m})$ with
$t \in [-\pi,0) \cup (0,\pi]$ then the inequality is strict:
$|((\alpha(z))^N)_{p,\tilde p}| < |(A^N)_{p,\tilde p}|$.
\end{lemma}

\begin{proof}
The first claim (i.e., the non strict inequality) follows from the fact that
$(\alpha^N)_{p,\tilde p}$ is a Laurent polynomial with natural coefficients.
For the second claim,
since $\psi$ is surjective, there exists even $N_0 > 4|\cD|$
such that there exist
tilings $\bt_0$ and $\bt_1$ of $\cR_{2|\cD|,N_0-2|\cD|;\emptyplug,\emptyplug}$ 
with $\psi(\bt_1) - \Tw(\psi_0) = 1$.
Use Lemma 3.1 in \cite{regulardisk}
to construct arbitrary tilings
$\bt_a$ and $\bt_b$ of the corks
$\cR_{0,2|\cD|;p,\emptyplug}$ and $\cR_{N_0-2|\cD|,N;\emptyplug,\tilde p}$,
respectively.
Concatenate $\bt_a$, $\bt_i$ and $\bt_b$ to define
$\tilde\bt_i \in \cT(\cR_{0,N;p,\tilde p})$.
The tilings $\tilde t_0$ and $\tilde t_1$ contribute with
$\tilde q^k$ and $\tilde q^{(k+m)}$ to 
$(\alpha^N)_{(p,\tilde p)}$, respectively,
proving that $(\alpha^N)_{(p,\tilde p)}$
is neither zero nor a monomial.
Moreover, for $z$ as in the statement,
$|z^k + z^{(k+m)}| < 2$, proving the strict inequality.
\end{proof}



\section{Proof of Theorem \ref{theo:normal}}
\label{section:normal}

We now prepare to prove Theorem \ref{theo:normal}.
Consider a fixed quadriculated disk $\cD$ and
a surjective homomorphism $\psi: G_{\cD} \to \ZZ$ with kernel
$K < G_{\cD}$ a normal subgroup of infinite index.
Recall from Equation \eqref{equation:P}
the definition of the (Laurent) polynomial $P_N \in  \ZZ[q,q^{-1}]$:
for $\cR_N = \cD \times [0,N]$,
\[ P_N(q) = \sum_{\bft \in \cT(\cR_N)} q^{\psi(\bft)}
= \sum_{k \in \ZZ} |\{ \bt \in \cT(\cR_N), \psi(\bt) = k\}|\,q^k,
\qquad
P_N(1) = |\cT(\cR_N)|. \]
Thus, for instance, if $\cD = [0,4]^2$, $\psi = \Tw$ and $N = 4$ then
\begin{align*}
P_4(q) &= 4413212553 + 310188792 (q + q^{-1}) + \\
&\qquad 8955822 (q^2 + q^{-2}) + 15144 (q^3 + q^{-3}) + 18 (q^4 + q^{-4}).
\end{align*}
For the same $\cD$, $\psi$ and $N = 60$,
$P_{60}$ has terms from $673511306237603716\,q^{-88}$
to $673511306237603716\,q^{88}$;
the largest coefficients are of the order of $10^{156}$
for $q^k$, $|k| \le 10$ (i.e., $k = \Tw(\bft)$),
and are shown in Figure \ref{fig:tw60}.

Our result can be considered a variant of the Central Limit Theorem
and its proof will be similar.
Let $\bT_N$ be a random tiling of $\cR_N$ (with uniform distribution);
we also consider the integer valued random variable
$\Psi_N = \psi(\bT_N)$.
We then have
\[ \EE\left(q^{\Psi_N}\right) = \frac{P_N(q)}{P_N(1)} 
= \sum_{k \in \ZZ} \Prob[\psi(\bT_N) = k]\,q^k, \qquad
\varphi_{\Psi_N}(t) = \frac{P_N(e^{it})}{P_N(1)}; \]
here $\varphi_{\Psi_N}$ is the characteristic function
of the random variable $\Psi_N$.
We prove that these characteristic functions,
with a minor adjustment,
converge uniformly over compacts to
a gaussian when $N$ goes to infinity:
this implies Theorem \ref{theo:normal}.

We introduce some more notation.
For $z \in \Ss^1 \subset \CC$,
the evaluation of the $\psi$-adjacency matrix $\alpha(z)$
(obtained by substituting $z$ for $\tilde q$)
is Hermitian and therefore diagonalizable with real spectrum
and unit complex orthogonal eigenvectors.
Recall that $\lambda_1 > 0$ is the eigenvalue of largest absolute value
of $A = \alpha(1)$.
Notice that $\alpha(\bar z) = \overline{\alpha(z)} = (\alpha(z))^{\top}$.
Let 
\[ 
\hat\alpha: \RR \to \CC^{\cP\times\cP}, \qquad
\hat\alpha(t) =
\frac{1}{\lambda_1}\;\alpha\left(\exp\left(\frac{it}{m}\right)\right); \]
let $\eta_1(t) \ge \eta_2(t) \ge \cdots \ge \eta_{|\cP|}(t)$
be the (real) eigenvalues of $\hat\alpha(t)$.

The functions $\eta_k: \RR \to \RR$ are clearly continuous.
Also, $\eta_1(0) = 1$ and,
from Lemma \ref{lemma:pf}, $\eta_k(0) \in (-1,1)$ for $k > 1$.
It follows from Rellich theorem
(see \cite{ReedSimon}, ch. 12 sec. 1 and \cite{Wimmer})
that the (real) eigenvalues of $\hat\alpha(t)$
can be locally defined
as (real) analytic functions of $t \in \RR$
(but not necessarily in order).
In particular, the restriction
$\eta_1: (-\epsilon,\epsilon) \to \RR$ is (real) analytic.
We need a little more information about the functions $\eta_k$,
particularly near $t = 0$.

\begin{lemma}
\label{lemma:contraction}
Let $\cD$ be a quadriculated disk
and $\psi: G_{\cD} \to \ZZ$ be a surjective homomorphism;
let $\eta_k: \RR \to \RR$, $1 \le k \le |\cP|$,
be the continuous maps above.
\begin{enumerate}
\item{ If $t \in [-\pi,\pi] \smallsetminus \{0\}$
then $\eta_k(t) \in (-1,1)$ (for all $k$).}
\item{ The function $\eta_1$
satisfies $\eta_1(0) = 1$, $\eta_1'(0) = 0$ and
$\eta_1''(0) < 0$.}
\item{ There exist $\epsilon > 0$, $c \in (0,1)$ and
constants $a_2 = -\eta_1''(0)/2 > 0$ and $a_4^{-} < a_4^{+}$ such that
$\eta_1$ is real analytic in $(-\epsilon,\epsilon)$ and
for all $t \in (-\epsilon,\epsilon)$ we have
\begin{equation}
\label{eq:tayloreta}
c < \exp(- a_2 t^2 + a_4^{-} t^4) \le \eta_1(t) \le 
\exp(- a_2 t^2 + a_4^{+} t^4) \le 1
\end{equation}
and, for $k \ne 1$, $\eta_k(t) \in (-c,c)$.}
\end{enumerate}
\end{lemma}


\begin{proof}
Let $v_1 \in (0,1)^{\cP} \subset \CC^{\cP}$
be the eigenvector of $\alpha(1) = A$ corresponding to $\lambda_1$,
as in Lemma \ref{lemma:pf}.
Consider the convex compact set $Y \subset \CC^{\cP}$
of all vectors $y \in \CC^{\cP}$
such that $|y_p| \le (v_1)_p$ (for all $p \in \cP$).
We claim that
$(\hat\alpha(t))^N[Y] \subset \interior(Y)$
if $t \in [-\pi,\pi] \smallsetminus \{0\}$ and
$N$ is sufficiently large;
notice that this claim implies the first item.

From the first claim of Lemma \ref{lemma:supercork},
if $y \in Y$ and $z = \exp(\frac{it}{m})$ then
\begin{align*}
|((\hat\alpha(t))^N y)_{p}|
&\le \frac{1}{\lambda_1^N} \sum_{\tilde p \in \cP}
|((\alpha(z))^N)_{(p,\tilde p)}| |y_{\tilde p}| \\
&\le \frac{1}{\lambda_1^N} \sum_{\tilde p \in \cP}
(A^N)_{(p,\tilde p)} (v_1)_{\tilde p}
= \frac{1}{\lambda_1^N} (A^N v_1)_p
= (v_1)_{p}
\end{align*}
and therefore $Y$ is invariant under $(\hat\alpha(t))^N$.
From the second claim of Lemma \ref{lemma:supercork},
if $N \ge N_0$ then some inequality in the display above is strict,
completing the proof of the first item.

We already saw that $\eta_1$ is even, real analytic
in a neighborhood of $t = 0$,
and that $\eta_1(0) = 1$ is a local maximum point.
In order to complete the proof of the second item
we are left with proving that $\eta_1''(0) < 0$.
We again consider the invariant set $Y$.
Take $N \ge N_0$ so that,
from Lemma \ref{lemma:supercork},
for all $p, \tilde p \in \cP$
there exists $k \in \ZZ$ such that
the coefficients of $\tilde q^k$ and $\tilde q^{(k+m)}$
in $(\alpha^N)_{p,\tilde p}$ are both positive.
There exists therefore $\tilde c > 0$ and $\tilde\epsilon > 0$ such that,
for all $t \in (-\tilde\epsilon,\tilde\epsilon)$ and
for all $p, \tilde p \in \cP$, we have
$|((\hat\alpha(t))^N)_{p,\tilde p}| \le
(1-\tilde c t^2) ((\hat\alpha(1))^N)_{p,\tilde p} $.
We thus have $(\hat\alpha(t))^N[Y] \subseteq (1-\tilde c t^2) Y$
and therefore, for all $k$,
$|\eta_k(t)|^N \le (1-\tilde c t^2)$;
taking $k = 1$ we obtain the second item.
The third item is then easy.
\end{proof}

\begin{proof}[Proof of Theorem \ref{theo:normal}]
Let $\eta_1(t)$ be the (simple and isolated)
largest eigenvalue of $\hat\alpha(t)$;
we already saw that the function
$\eta_1: (-\epsilon,\epsilon) \to (c,1] \subset \RR$ is real analytic.
For $t \in  (-\epsilon,\epsilon)$,
let $v_1(t)$ be the corresponding unit eigenvector;
the function $v_1: (-\epsilon,\epsilon) \to \CC^{\cP\times\cP}$
is also real analytic.
Notice that $v_1(0)$ is the eigenvector of $A = \alpha(1)$
previously denoted by $v_1$.
Let $\Pi_1(t) = v_1(t)(v_1(t))^\ast$ be the orthogonal projection
onto the subspace spanned by $v_1(t)$.
Write $\hat\alpha(t) = \eta_1(t) \Pi_1(t) + C_1(t)$
so that $\Pi_1(t) C_1(t) = C_1(t) \Pi_1(t) = 0$
and $C_1(t)$ is a $c$-contraction.
Thus
\[ (\hat\alpha(t))^N = (\eta_1(t))^N \Pi_1(t) + (C_1(t))^N
=  (\eta_1(t))^N \Pi_1(t) + o(1), \]
where the error term $o(1)$ goes to $0$ exponentially
(when $N$ goes to $+\infty$).
Let 
\[ p_1(t) =
\frac{(\Pi_1(t))_{\emptyplug,\emptyplug}}{(\Pi_1(0))_{\emptyplug,\emptyplug}}
= \left| \frac{(v_1(t))_{\emptyplug}}{(v_1(0))_{\emptyplug}} \right|^2; \]
the real function $p_1$ is even and real analytic with $p_1(0) = 1$
(notice that the denominator
is known to be positive from Lemma \ref{lemma:pf}).

Recall (again from Lemma \ref{lemma:pf}) that
\[ P_N(1) = |\cT(\cR_N)| =
\lambda_1^N ((v_1)_{\emptyplug}) (1+o(1)). \]
Thus, when $N$ goes to infinity,
\[ \varphi_{\Psi_N}(t) =
\frac{P_N(e^{it})}{P_N(1)} = 
\frac{((\hat\alpha(t))^N)_{\emptyplug,\emptyplug}}%
{((\hat\alpha(0))^N)_{\emptyplug,\emptyplug}} 
= (1 + o(1)) (\eta_1(t))^N p_1(t); \] 
here again the error term $o(1)$ goes to $0$ exponentially in $N$
and uniformly in $t \in (-\epsilon,\epsilon)$.
We now estimate the characteristic function
\( \varphi_{{\Psi_N}/{\sqrt{N}}}(t) =
\varphi_{\Psi_N}({t}/{\sqrt{N}}) \).
There are constants $b_2, b_4^{+}, b_4^{-}$ such that
\[ \exp(b_2 t^2 + b_4^{-} t^4) \le p_1(t) \le \exp(b_2 t^2 + b_4^{+} t^4). \]
Together with Equation \eqref{eq:tayloreta} (in Lemma \ref{lemma:contraction})
we thus have
\[ (\eta_1(t/\sqrt{N}))^N p_1(t/\sqrt{N}) \approx \exp(-a_2t^2), \]
or, more precisely:
\begin{align*}
(\eta_1(t/\sqrt{N}))^N p_1(t/\sqrt{N}) &\le
\exp(-a_2 t^2) \;
\exp\left(\frac{b_2}{N} t^2 +
\frac{a_4^{+}}{N} t^4 + \frac{b_4^{+}}{N^2} t^4\right); \\
(\eta_1(t/\sqrt{N}))^N p_1(t/\sqrt{N}) &\ge
\exp(-a_2 t^2) \;
\exp\left(\frac{b_2}{N} t^2 +
\frac{a_4^{-}}{N} t^4 + \frac{b_4^{-}}{N^2} t^4\right).
\end{align*}
It follows that the sequence of characteristic functions
$(\varphi_{{\Psi_N}/{\sqrt{N}}})_{N \in \NN}$
converges uniformly in compacts to the gaussian $\exp(-a_2 t^2)$,
completing the proof.
\end{proof}

Given a nontrivial quadriculated disk $\cD$,
a surjective homomorphism $\psi: G_{\cD} \to \ZZ$ and $K = \ker(\psi)$,
construct $\xi: \cP^K \to \RR$ as in Section \ref{section:cocycle}.
For $p_0, p_N \in \cP$ and a tiling $\bt \in \cT(\cR_{0,N;p_0,p_N})$,
interpret $\bt$ as a path in $\cC_{\cD}$ and lift it to $\cC_{\cD}^K$,
thus defining a finite sequence $(\tilde p_k)_{0 \le k \le N}$
of vertices of $\cP^K$. Define
\[ \psi_{\xi}(\bt) = \xi(\tilde p_N) - \xi(\tilde p_0). \]
Notice that if $p_0 = p_N = \emptyplug$ then
$\psi_{\xi}(\bt) = \psi(\bt)$;
in other words, we are extending the definition of $\psi$
to tilings of corks.

\begin{coro}
\label{coro:xnormal}
Consider a nontrivial quadriculated disk $\cD$ and
a surjective homomorphism $\psi: G_{\cD} \to \ZZ$;
construct $\xi$ as above.
Given two fixed plugs $p,\tilde p \in \cP$,
let $\bT$ be a random tiling of $\cR_{0,N;p,\tilde p}$.
As $N \to \infty$, the real random variable
$(1/\sqrt{N})\psi_{\xi}(\bT)$
converges in distribution to a normal distribution centered at $0$.
Furthermore, the limit distribution does not depend on the choices
of $\xi$, $p$ or $\tilde p$.
\end{coro}

\begin{proof}
The result follows from Theorem \ref{theo:normal}
together with Lemma \ref{lemma:pf}.
Alternatively, it follows from the proof of Theorem \ref{theo:normal}.
\end{proof}


\section{Proof of Theorem \ref{theo:truenormal}}
\label{section:truenormal}

Our aim is to prove Theorem \ref{theo:truenormal}
(using Theorem \ref{theo:normal}):
it follows directly from Lemmas \ref{lemma:normallimsup}
and \ref{lemma:normalliminf} below.
The quadriculated disk $\cD$ and
the surjective homomorphism $\psi: G_{\cD} \to \ZZ$ will be fixed.
The constants $C_0, C_1 \in (0,+\infty)$ are as in Theorem \ref{theo:normal}.
We shall first introduce some (local) notation.

It follows from the surjectivity of $\psi$ that
there exist $N_{\bullet} \in \NN^\ast$
and tilings $\bt_{\bullet,0}, \bt_{\bullet,1}$ of $\cR_{N_{\bullet}}$
with $\psi(\bt_{\bullet,1}) = 1 + \psi(\bt_{\bullet,0})$.
For the rest of the section,
$N_{\bullet}$ and $\bt_{\bullet,i}$ will be fixed.

For a tiling $\bt$ of $\cR_N$,
let $B(\bt)$ be the set of $k \in \ZZ$, $0 \le k \le (N/N_{\bullet}) - 1$,
with the following properties:
\begin{itemize}
\item{the plugs $p_{kN_{\bullet}}$ and $p_{(k+1)N_{\bullet}}$ of $\bt$
both equal $\emptyplug$;}
\item{the restriction of $\bt$ to $\cR_{kN_{\bullet},(k+1)N_{\bullet}}$
is a translated copy of either $\bt_{\bullet,0}$ or $\bt_{\bullet,1}$.}
\end{itemize}
Let $\beta(\bt) = |B(\bt)|$.
Imitate Lemma \ref{lemma:manyverticalfloors} to obtain
constants $C_{\bullet}, c_{\bullet} \in (0,1)$ such that, for large $N$,
we have $\Prob[\beta(\bT_N) \le C_{\bullet}N] = o(c_{\bullet}^N)$
(where $\bT_N$ is a random tiling of $\cR_N$).
Define also an equivalence relation $\equiv$ in $\cT(\cR_N)$
(not to be confused with $\approx$ or $\sim$).
We have $\bt_0 \equiv \bt_1$ if and only if:
\begin{itemize}
\item{$B(\bt_0) = B(\bt_1)$;}
\item{if $k \notin B(\bt_0)$ then the restrictions of $\bt_0$ and $\bt_1$
to $\cR_{kN_1,(k+1)N_1}$ are equal.}
\end{itemize}
In other words, $\bt_0 \equiv \bt_1$ if and only if
$\bt_1$ is obtained from $\bt_0$
by substituting $\bt_{\bullet,i}$ for $\bt_{\bullet,1-i}$
in some of the blocks $\cR_{kN_\bullet,(k+1)N_\bullet}$,
$k \in B(\bt_0)$.
With this equivalence relation,
$\cT(\cR_N)$ is partitioned into equivalence classes.
The $\equiv$-equivalence class $[\bt]$ of $\bt$ has size $2^{\beta(\bt)}$.
Let $e(\bt) = \EE(\psi(\bT)\,|\,\bT \equiv \bt)$,
the average value of $\psi$ in the equivalence class $[\bt]$
of the tiling $\bt$.

Let $\bt$ be a tiling, $\psi(\bt) = \tau$.
Let $k_{i}$ be the number of $\bt_{\bullet,i}$ blocks in $\bt$,
so that $k_{0} + k_{1} = \beta(\bt)$.
In the equivalence class $[\bt]$,
the random variable $\psi(\bT)$ follows a binomial distribution:
\begin{equation}
\label{equation:binomial}
\Prob[\psi(\bT) = \tau+j \,|\, \bT \equiv \bt] =
\begin{cases}
2^{-\beta(\bt)} \binom{\beta(\bt)}{j + k_{1}}, & -k_{1} \le j \le k_{0}, \\
0, & \textrm{otherwise.}
\end{cases}
\end{equation}

\begin{lemma}
\label{lemma:normallimsup}
Let $\cD \subset \RR^2$ be a quadriculated disk;
let $\phi: G_{\cD} \to \ZZ$ be a surjective homomorphism;
let $C_0, C_1 \in (0,+\infty)$ be as in Theorem \ref{theo:normal}.
Let $(t_N)$ be a sequence of integers with
$\lim_{N \to \infty} t_N/\sqrt{N} = \tau_0 \in \RR$.
We then have
\[ \limsup_{N \to \infty} \sqrt{N}\,\Prob[\phi(\bT_N) = t_N] \le
C_0 \exp(-C_1 \tau_0^2). \]
\end{lemma}

We first need a lemma about binomial numbers.
We follow the convention $\binom{n}{b} = 0$
if $b < 0$ or $b > n$.

\begin{lemma}
\label{lemma:binomiallimsup}
For every $\epsilon > 0$ there exist $\delta_\epsilon > 0$
with the following property.
Given $\delta \in (0,\delta_\epsilon)$ 
there exists $n_\delta \in \NN$ such that
if $n > n_\delta$, and $a \in \ZZ$, then
\[ \binom{n}{a} \le (1+\epsilon)\,\frac{1}{2\delta\sqrt{n}}\,
\sum_{a-\delta\sqrt{n} < b < a+\delta\sqrt{n}} \binom{n}{b}. \]
Furthermore, given a compact interval $K \subset (0,\delta_\epsilon)$,
the value of $n_\delta$ can be taken to be the same for all $\delta \in K$.
\end{lemma}

\begin{proof}
The inequality is trivial if $a < 0$ or $a > n$.
Assuming $0 \le a \le n$, notice first that
\[ \Delta = \binom{n}{a-1} - 2 \binom{n}{a} + \binom{n}{a+1} =
\binom{n}{a}\,\frac{(n-2a)^2 - n - 2}{(a+1)(n-a+1)} \]
and therefore $\Delta > 0$ for $(n-2a)^2 > n+2$.
The letter $\Delta$ is used to remind us of the Laplacian:
$\Delta$ is a discrete second derivative.
The concavity of the graph of $\binom{n}{x}$
(as a function of $x$) thus points down
in the central interval $|a - \frac{n}{2}| \le \sqrt{n+2}$
and points up elsewhere.
Thus, if $J \subset \ZZ$ is an interval
disjoint from the central interval
and $a \in J$ is the central point then
\[ \binom{n}{a} \le \frac{1}{|J|}\,\sum_{b \in J} \binom{n}{b}. \]
Assuming $\delta_\epsilon < 1$ and $n$ large, this takes care of the case
$|a - \frac{n}{2}| \ge 3\sqrt{n}$.

For the case $|a - \frac{n}{2}| \le 3\sqrt{n}$
we also assume $\delta_\epsilon < 1$
so that $|b - \frac{n}{2}| \le 4\sqrt{n}$ for all $b$ in the summation.
Stirling approximation formula yields
\[ \frac{1}{2^n} \binom{n}{b} =
(1+o(n))\,\sqrt{\frac{2}{\pi n}}\,
\exp\left(-\frac{2(b-\frac{n}{2})^2}{n}\right); \]
assuming the conditions above on $a$ and $b$,
the approximation holds uniformly.
Both this equation or the Central Limit Theorem yield
\[ \frac{1}{2^n}\,
\sum_{a-\delta\sqrt{n} < b < a+\delta\sqrt{n}} \binom{n}{b} =
(1 + o(n)) \, \sqrt{\frac{2}{\pi}} \,
\int_{t_0-\delta}^{t_0+\delta} \exp(-2t^2) dt,
\qquad t_0 = \frac{a-\frac{n}{2}}{\sqrt{n}}; \]
assuming $\delta \in K$,
$K \subset (0,1)$ a fixed compact interval,
this also holds uniformly.
Set 
\[ A = \binom{n}{a}, \qquad B = \frac{1}{2\delta\sqrt{n}}\,
\sum_{a-\delta\sqrt{n} < b < a+\delta\sqrt{n}} \binom{n}{b}. \]
We have
\[ \frac{A}{B} = (1+o(n))
\frac{\exp(-2t_0^2)}{\frac{1}{2\delta}
\int_{t_0-\delta}^{t_0+\delta} \exp(-2t^2) dt}. \]
Given $\epsilon > 0$, there exists $\delta_\epsilon > 0$ such that
the fraction on the right hand side lies in the interval
$(1-\frac{\epsilon}{2}, 1+\frac{\epsilon}{2})$
for all $t_0 \in [-4,4]$ and all $\delta \in (0,\delta_\epsilon)$,
completing the proof of the lemma.
\end{proof}

\begin{proof}[Proof of Lemma \ref{lemma:normallimsup}]
Let $(t_N)$ be a sequence of integers, as in the statement,
with $\lim_{N\to\infty} t_N/\sqrt{N} = \tau_0 \in \RR$.
We prove that for any $\epsilon \in (0,\frac{1}{10})$ we have
\[ \limsup_{N \to \infty} \sqrt{N}\,\Prob[\phi(\bT_N) = t_N] \le
(1+10\epsilon)\,C_0 \exp(-C_1 \tau_0^2). \]
This will prove our lemma.

Given $\delta > 0$,
let $\tau^{\pm} = \tau_0 \pm \delta$.
With $\tau_0$ fixed, for sufficiently small $\delta > 0$:
\begin{equation}
\label{equation:limsup1}
\frac{1}{2\delta} 
\int_{\tau^{-}}^{\tau^{+}} C_0 \exp(-C_1\tau^2) d\tau <
(1+{\epsilon})\,C_0 \exp(-C_1 \tau_0^2). 
\end{equation}
Apply Theorem \ref{theo:normal} to deduce that,
for sufficiently large $N$,
\begin{equation}
\label{equation:limsup2}
\Prob[ \tau^{-} \sqrt{N} <  \psi(\bT_N) < \tau^{+} \sqrt{N} ] <
(1+ \epsilon)
\int_{\tau^{-}}^{\tau^{+}} C_0 \exp(-C_1\tau^2) d\tau. 
\end{equation}
Let $\bt_0$ be a tiling of $\cR_N$
with $\beta(\bt_0) > C_{\bullet} N$.
Recall from Equation \eqref{equation:binomial}
that $\Prob[\psi(\bT_N) = t\,|\,\bT_N\equiv \bt_0]$
follows a binomial distribution.
We may therefore apply Lemma \ref{lemma:binomiallimsup}
to deduce that
\begin{gather*}
\Prob[\psi(\bT_N) = t_N\,|\, \bT_N \equiv \bt_0] \le \qquad\qquad\qquad\\
\qquad\qquad\qquad \le \frac{1+2\epsilon}{2\delta\sqrt{N}}\,
\Prob[ \tau^{-} \sqrt{N} <  \psi(\bT_N) < \tau^{+} \sqrt{N} \,|\,
\bT_N \equiv \bt_0].
\end{gather*}
Since this holds for any such $\bt_0$,
and since the total measure of the equivalence classes
with $\beta \le C_{\bullet}N$ is very small, we have
\begin{equation}
\label{equation:limsup3}
\Prob[\psi(\bT_N) = t_N] \le 
\frac{1+3\epsilon}{2\delta\sqrt{N}}\,
\Prob[ \tau^{-} \sqrt{N} <  \psi(\bT_N) < \tau^{+} \sqrt{N} ].
\end{equation}
Put together Equations \eqref{equation:limsup1},
\eqref{equation:limsup2} and \eqref{equation:limsup3}
to obtain the desired conclusion.
\end{proof}

\begin{lemma}
\label{lemma:normalliminf}
Let $\cD \subset \RR^2$ be a quadriculated disk;
let $\phi: G_{\cD} \to \ZZ$ be a surjective homomorphism;
let $C_0, C_1 \in (0,+\infty)$ be as in Theorem \ref{theo:normal}.
Let $(t_N)$ be a sequence of integers with
$\lim_{N \to \infty} t_N/\sqrt{N} = \tau_0 \in \RR$.
We then have
\[ \liminf_{N \to \infty} \sqrt{N}\,\Prob[\phi(\bT_N) = t_N] \ge
C_0 \exp(-C_1 \tau_0^2). \]
\end{lemma}

\begin{proof}
For tilings of $\cR_N$, we consider the real variable
$\tau = \psi(\bt)/\sqrt{N} \in \RR$.
We know from Theorem \ref{theo:normal} that
the random variable $\psi(\bT)/\sqrt{N}$
converges in distribution to the gaussian
$g(\tau) = C_0 \exp(-C_1 \tau^{2})$ in the $\tau$-axis.
If we restrict ourselves to an equivalence class $[\bt]$
we have of course a binomial distribution.
Let $\tilde\beta = \beta(\bt)/N$ and
$\tilde e(\bt) = e(\bt)/\sqrt{N}$;
if we exclude a subset of $\cT(\cR_N)$ of very small measure
we may assume that $C_{\bullet} \le \tilde\beta \le 1/N_{\bullet}$.
The distribution of the random variable $\psi(\bT)/\sqrt{N}$
in the equivalence class $[\bt]$
thus also approaches (at least in distribution)
a gaussian
$g_{\bt}(\tau) = C_{0,\bt} \exp(-C_{1,\bt} (\tau-\tilde e(\bt))^{2})$
in the $\tau$-axis.
Notice that the logarithmic derivative 
\[ \frac{g'_{\bt}(\tau)}{g_{\bt}(\tau)}
= -2C_{1,\bt} (\tau - \tilde e(\bt)) \]
is an affine strictly decreasing function.
The condition $C_{\bullet} \le \tilde\beta \le 1/N_{\bullet}$
implies
$C_{1,\min} \le C_{1,\bt} \le C_{1,\max}$.
Here $C_{1,\min}, C_{1,\max}$ are constants:
more precisely, they are functions of $\cD$ only,
not of $N$ or of $\bt$.
Thus, the slopes of the logarithmic derivatives $g'_{\bt}/g_{\bt}$
are uniformly bounded
(excepting the set of very small measure
which we are consistently neglecting).

Assume without loss of generality that $\tau_0 \ge 0$.
Given $\epsilon \in (0,\frac{1}{1000})$,
choose $\lambda > 1$ such that
$\lambda^2 \in (1,1+\epsilon)$.
Choose $k \in \NN^\ast$, $k > 4$, with $\lambda^{(-k+2)/2} < \epsilon$.

If $\delta > 0$ is small then $g(\tau_0+\delta) > (1-\epsilon)\,g(\tau_0)$.
Also, if $\delta > 0$ is taken sufficiently small then
$C_{1,\min} \le C_{1,\bt} \le C_{1,\max}$ and
$g'_{\bt}(\tau)/g_{\bt}(\tau) > \log(\lambda)/\delta$
imply
$g'_{\bt}(\tau+(k-2)\delta)/g_{\bt}(\tau+(k-2)\delta) >
\log(\lambda)/(2\delta)$.
Choose such a small $\delta > 0$.

For $-2 \le i \le k$, set $\tau_i = \tau_0 + i\delta$
and $J_i = [\tau_{i-1},\tau_i] \subset \RR$.
From Theorem \ref{theo:normal}, for sufficiently large $N$ we have
\[ 1-\epsilon <
\frac{\Prob[\psi(\bT_N)/\sqrt{N} \in J_i]}{\int_{J_i} g(\tau) d\tau}
< 1+\epsilon. \]
For such $N$, consider the equivalence classes $[\bt]$
which satisfy $C_\bullet N \le \beta(\bt) \le N/N_\bullet$.
Such an equivalence class $[\bt]$ 
(and by extension, the tiling $\bt$) is of {\em type I} if 
$g_{\bt}(\tau_1)/g_{\bt}(\tau_0) \le \lambda$
and of {\em type II} otherwise.

If $[\bt]$ is of type II then
$g'_{\bt}(\tau_0)/g_{\bt}(\tau_0) > \log(\lambda)/\delta$
so that
$g'_{\bt}(\tau_{i-1})/g_{\bt}(\tau_{i-1}) > \log(\lambda)/(2\delta)$
for all $i \le k-1$.
We thus have
$g_{\bt}(\tau_i)/g_{\bt}(\tau_{i-1}) > \sqrt{\lambda}$
for all such $i$
and therefore
$g_{\bt}(\tau_{k-1})/g_{\bt}(\tau_1) > \lambda^{(k-2)/2} > 1/\epsilon$.
We thus have
\[ \int_{J_1} g_{\bt}(\tau) d\tau <
\epsilon\,\int_{J_k} g_{\bt}(\tau) d\tau. \]
It follows that, for each $[\bt_0]$ of type II,
\[ \left|\left\{ \bt \equiv \bt_0,
\tau_0 < \frac{\psi(\bt)}{\sqrt{N}} < \tau_1 \right\}\right|
< \frac{3\epsilon}{2}\,
\left|\left\{ \bt \equiv \bt_0,
\tau_{k-1} < \frac{\psi(\bt)}{\sqrt{N}} < \tau_k \right\}\right|. \]
Thus, for sufficiently large $N$,
\begin{gather*}
\left|\left\{ \bt \in \cT(\cR_N) \,|\,
\psi(\bt)/\sqrt{N} \in J_1, 
\bt \textrm{ of type II} \right\}\right|
< \qquad \qquad \\
< 2\epsilon
\left|\left\{ \bt \in \cT(\cR_N) \,|\,
\psi(\bt)/\sqrt{N} \in J_k, 
\bt \textrm{ of type II} \right\}\right| < \\
\qquad\qquad < 3\epsilon\,
|\cT(\cR_N)|\,\int_{J_k} g(\tau) d\tau
< 3\epsilon\,
|\cT(\cR_N)|\,\int_{J_1} g(\tau) d\tau. 
\end{gather*}
We therefore have
\[
\left|\left\{ \bt \in \cT(\cR_N) \,|\,
\psi(\bt)/\sqrt{N} \in J_1, 
\bt \textrm{ of type I} \right\}\right|
> (1-4\epsilon)\, 
|\cT(\cR_N)|\,\int_{J_1} g(\tau) d\tau. \]
Consider now a class $[\bt_0]$ of type I,
so that $g_{\bt_0}(\tau_0) \ge (1/\lambda)\,g_{\bt_0}(\tau_1)$
and  therefore $g_{\bt_0}(\tau_0) > (1-\epsilon) g_{\bt_0}(\tau)$
for all $\tau \in J_1 = [\tau_0,\tau_1]$.
We therefore have that, for sufficiently large $N$,
if $t \in \ZZ$, $\tau_0 \le t/\sqrt{N} \le \tau_1$ then
\[
\left|\left\{ \bt \in \cT(\cR_N) \,|\,
\bt \equiv \bt_0,
\psi(\bt) = t_N 
\right\}\right|
>
(1-2\epsilon)\,
\left|\left\{ \bt \in \cT(\cR_N) \,|\,
\bt \equiv \bt_0,
\psi(\bt) = t 
\right\}\right|
\]
and therefore
\[
\left|\left\{ \bt \in \cT(\cR_N) \,|\,
\bt \equiv \bt_0,
\psi(\bt) = t_N 
\right\}\right|
>
\frac{1-3\epsilon}{\delta\,\sqrt{N}}\,
\left|\left\{ \bt \in \cT(\cR_N) \,|\,
\bt \equiv \bt_0,
\frac{\psi(\bt)}{\sqrt{N}} \in J_1
\right\}\right|.
\]
Adding over all classes of type I we have
\begin{align*}
\left|\left\{ \bt \in \cT(\cR_N) \,|\,
\psi(\bt) = t_N 
\right\}\right|
&>
\frac{1-3\epsilon}{\delta\,\sqrt{N}}\,
\left|\left\{ \bt \in \cT(\cR_N) \,|\,
\frac{\psi(\bt)}{\sqrt{N}} \in J_1,
\bt \textrm{ of type I}
\right\}\right| \\
&>
\frac{1-8\epsilon}{\delta\,\sqrt{N}}\,
|\cT(\cR_N)|\,\int_{J_1} g(\tau) d\tau
> 
\frac{1-8\epsilon}{\sqrt{N}}\,
|\cT(\cR_N)|\,g(\tau_0).
\end{align*}
Since this estimate holds for any $\epsilon > 0$
we are done.
\end{proof}


\section{Proof of Theorems \ref{theo:components},
\ref{theo:limitprob} and \ref{theo:trit}}
\label{section:components}

The proof of Theorems \ref{theo:components},
\ref{theo:limitprob} and \ref{theo:trit} now mostly amount
to putting together previous results.
We begin by (essentially) restating Lemma 11.1 from \cite{regulardisk}.

\begin{lemma}
\label{lemma:cD}
\cite{regulardisk}
Let $\cD \subset \RR^2$ be a nontrivial quadriculated disk.
There exist $c_{\cD} \in \QQ$ and $d \in \RR$ with the properties below.
\begin{enumerate}
\item{If $\bt$ is a tiling of $\cR_N$ then $|\Tw(\bt)| \le c_{\cD} N$.}
\item{If $t \in \ZZ$, $|t| \le c_{\cD} N - d$ then
there exists a tiling $\bt$ of $\cR_N$ with $\Tw(\bt) = t$.}
\end{enumerate}
\end{lemma}

\begin{proof}[Proof of Theorem \ref{theo:components}]
The first item of Theorem \ref{theo:components}
is a repetition of the first item of Lemma \ref{lemma:cD}.

For the second item, take $b = d+c_{\cD}M$
(where $M$ is as in Theorem \ref{theo:M} and
in the definition of fat components).
From Lemma \ref{lemma:cD}, there exists a tiling $\bt_0$ of $cR_{N-M}$
with $\Tw(\bt_0) = t$.
The tiling $\bt_0 \ast \bt_{\nvert,M}$ belongs
to a fat component $\cF_{N,t}$.

In order to prove uniqueness, consider a fat component $\tilde\cF$
with $\Tw(\tilde\cF) = t$.
By definition, there exists a tiling of the form
$\bt_1 \ast \bt_{\nvert,M} \in \tilde\cF$.
We have that $\bt_1$ is a tiling of $\cR_{N-M}$
with $\Tw(\bt_1) = \Tw(\bt_0)$. 
Since $\cD$ is regular, $\bt_0 \sim \bt_1$.
It follows from Theorem \ref{theo:M} that
$\bt_0  \ast \bt_{\nvert,M} \approx \bt_1  \ast \bt_{\nvert,M}$
and therefore $\tilde\cF = \cF_{N,t}$, as desired.
The third item is similar.

If $\bt$ belongs to a thin component,
we have $\nvert(\bt) < M$.
The estimate follows from Lemma \ref{lemma:manyverticalfloors}.
\end{proof}

\begin{proof}[Proof of Theorem \ref{theo:limitprob}]
From Theorem \ref{theo:truenormal},
the probability that $\Tw(\bT_0) = \Tw(\bT_1)$
is more than $(1-\epsilon) C_0^2/N$
(the probability that both twists equal $0$).
This is important because we will neglect much smaller probabilities.

Assume first that $K = \ker(\Tw) < G^{+}_{\cD}$ is finite
with $k = |K|$.
Choose (and keep) $N_0 \in 2\NN^\ast$ and tilings
$\bt_0, \ldots, \bt_{k-1}$ of $\cR_{N_0}$
representing the distinct elements of $K$.
In particular, $\Tw(\bt_i) = 0$ (for all $i$).
A {\em block} in a tiling $\bt$ is a translated copy
of one of the $\bt_i$
so that $\bt = \bt_{-} \ast \bt_{i} \ast \bt_{+}$.
By imitating Lemma \ref{lemma:manyverticalfloors},
the probability that a random tiling $\bT$ does not include
at least one block is $o(c^N)$ (for some $c \in (0,1)$)
and therefore negligible.
We define an equivalence relation:
$\bt \equiv \tilde\bt$ if and only if
$\bt = \bt_{-} \ast \bt_{i} \ast \bt_{+}$,
$\tilde\bt = \bt_{-} \ast \bt_{j} \ast \bt_{+}$,
and $\bt_{-}$ admits no block.
In other words, $\bt \equiv \tilde\bt$
if and only if $\tilde\bt$ is otained from $\bt$
by replacing the first block $\bt_i$
by some other block $\bt_j$.
With the exception of tilings which admit no block,
all equivalence classes have size $k$.
Notice that $\bt \equiv \tilde\bt$ implies
$\Tw(\bt) = \Tw(\tilde\bt)$.

Assume that $\bT_0$ has been selected first.
In each $\equiv$-equivalent class with $\Tw(\cdot) = \Tw(\bT_0)$
there exists precisely one value of $\bT_1$
with $\bT_0 \sim \bT_1$.
We already know that the probability
that $\bT_0 \sim \bT_1$ and $\bT_0 \not\approx \bT_1$
is negligible.
The desired probability tends to $1/k$, as claimed.

Assume now that $K$ is infinite.
Consider $k \in \NN^\ast$ and select $N_0$
and tilings $\bt_0, \ldots, \bt_{k-1}$ of $\cR_{N_0}$
representing $k$ distinct elements of $K$.
Define the equivalence relation and equivalence classes as above.
In each equivalence class there exists at most
one value of $\bT_1$ with $\bT_0 \sim \bT_1$.
This implies that the $\limsup$ of the desired probability
is at most $1/k$.
Since this holds for any $k$,
the limit exists and is equal to $0$, as claimed.
\end{proof}

We do not know if there exists a quadriculated disk $\cD$
which is not regular but for which $K$ is finite.

\begin{proof}[Proof of Theorem \ref{theo:trit}]
Take $\tilde b = b - 4c$.
The tiling $\bt_0$ is constructed as follows.
The first $4$ (or $2$) floors contain a copy of
one of the two tilings in the left half of Figure \ref{fig:trit2}
(the other dominoes being vertical).
The remaining $N-4$ floors 
contain a tiling in $\cF_{N-4,t_0}$.
Notice that Theorem \ref{theo:components}
(and the choice of $\tilde b$)
guarantee that such a tiling exists.
The tiling $\bt_1$ is obtained by performing the trit
indicated in the left half of Figure \ref{fig:trit2}.
The two resulting tilings are in $\cF_{N,t_0}$ and $\cF_{N,t_0+1}$,
as desired;
indeed, we may assume they have $M$ vertical floors
(where $M$ is as in Theorem \ref{theo:M}
and as in the definition of a fat component).

The giant component $\cG$
includes all fat components $\cF_{N,t}$, $|t| \le cN - \tilde b$,
and all tiling with $M+4$ vertical floors.
The estimate on the size of the complement follows
from Lemma \ref{lemma:manyverticalfloors}.
\end{proof}

\begin{figure}[ht]
\begin{center}
\includegraphics[scale=0.27]{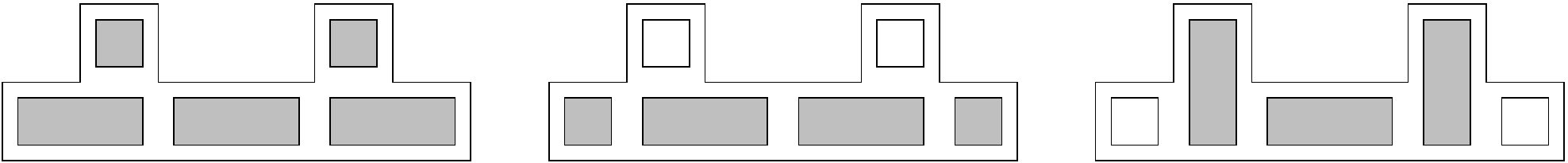}
\end{center}
\caption{A tiling which admits no flips or trits.}
\label{fig:notrit}
\end{figure}

There appear to be examples
where the giant connected component $\cG$
under flips and trits is actually equal to $\cT(\cR_N)$.
This is true for the boxes $[0,3]^2 \times [0,2]$ and $[0,4]^3$.
This is not true for the rather special
region shown in Figure \ref{fig:notrit}
(from \cite{segundoartigo}).
It would be interesting to clarify when we have $\cG = \cT(\cR_N)$.


\section{Final remarks}
\label{section:final}

The most obvious open question is Conjecture \ref{conj:main}.
In the introduction, it is formulated for boxes,
but, if true, there are probably similar statements
for regions of other shapes.
Also, the set $\cT(\cR)$ of tilings should have
fat connected components $\cF_{t} \subset \cT(\cR)$
for different values of the twist: $\Tw(\cF_t) = t$.

It would also be interesting to have a better understanding
of the thin connected components,
both in the case of cylinders and for more general regions.
There is some experimental evidence which hints
that they should be very small.

For cylinders, a  better understanding of which quadriculated disks
are regular would be very helpful.
A study of the irregular cases would also be desirable.

The main results of this paper give us significant information
concerning the distribution of $\Tw(\bT)$
for $\bT$ a random tiling of a cylinder.
Many natural questions remain unanswered, however.
Is the sequence $\Prob[\Tw(\bT) = t]$ unimodal?

The question which is easiest to formulate
is whether flips and trits make $\cT(\cR)$ connected
for $\cR$ a box.



\medskip

\noindent
\footnotesize
Departamento de Matem\'atica, PUC-Rio \\
Rua Marqu\^es de S\~ao Vicente, 225, Rio de Janeiro, RJ 22451-900, Brazil \\
\url{saldanha@puc-rio.br}

\end{document}